\documentclass[a4paper,reqno,11pt]{amsart}
\usepackage{amsfonts,amsmath,amsthm,amssymb,stmaryrd}
\usepackage{color}
\allowdisplaybreaks[4]

\newtheorem{Theorem}{Theorem}[section]

\newtheorem{Lemma}[Theorem]{Lemma}
\newtheorem{Proposition}{Proposition}[section]
\newtheorem{Remark}{Remark}[section]

\newcommand{\beq}{\begin{equation}}
\newcommand{\eeq}{\end{equation}}
\newcommand{\ben}{\begin{eqnarray}}
\newcommand{\een}{\end{eqnarray}}
\newcommand{\beno}{\begin{eqnarray*}}
\newcommand{\eeno}{\end{eqnarray*}}

\newcommand{\pt}{\partial_{t}}
\newcommand{\px}{\partial_{x}}
\newcommand{\py}{\partial_{y}}
\newcommand{\pz}{\partial_{z}}
\newcommand{\ptau}{\partial_{\tau}}

\newcommand{\ty}{\infty}

\newcommand{\R}{\mathbb{R}}
\newcommand{\N}{\mathbb{N}}

\newcommand{\T}{\mathbb{T}}

\newcommand{\de}{\delta}

\newcommand{\la}{\lambda}
\newcommand{\ga}{\gamma}
\newcommand{\al}{\alpha}
\newcommand{\si}{\sigma}

\newcommand{\mH}{\mathcal{H}}

\newcommand{\ep}{\epsilon}

\def\p{\partial}

\newcommand{\Om}{\Omega}
\newcommand{\om}{\omega}

\numberwithin{equation}{section} \allowdisplaybreaks
\usepackage[left=1 in, right=1 in,top=1 in, bottom=1 in]{geometry}
\begin{document}
\title[Local well-posedness of the boundary layer]{\bf Local well-posedness of the boundary layer for a pseudo-plastic fluid by energy methods}
\author{Zhonger Wu$^{1}$}
\author{Zhong Tan$^{2,*}$}
\thanks{$^{1}$ Department of Mathematics, Shantou University, Shantou 515063, China }
\thanks{$^{2}$ School of Mathematical Sciences, Xiamen University, Xiamen, 361005, China }
\thanks{$^{*}$Corresponding author: Zhong Tan, tan85@xmu.edu.cn}
\thanks{Zhonger Wu, wze622520@163.com}

\begin{abstract}
We study the well-posedness of the boundary layer for a pseudo-plastic fluid by energy methods under Oleinik's monotonicity assumption. We need to address two main difficulties: derivative loss and the difficulty caused by viscosity. For the first difficulty, we borrow the cancellation mechanism proposed by Masmoudi and Wong [CPAM, 2015]. For the second difficulty, we combine the monotonicity assumption and Fa\`a di Bruno formula to obtain precise control over the high-order derivatives of viscosity, which is the main contribution of this paper.


\end{abstract}
\maketitle

\section{Introduction}
We consider the well-posedness of the pseudo-plastic fluid boundary layer equations in a periodic domain $\{(t,x,y)|t\in[0,T],x\in\T,y\in\R_+\}$:
\begin{equation}\label{feiniu}
\left\{\begin{array}{l}
\pt u +u\px u +v\py u-\py(|\py u|^{n-1}\py u)+\px p=0,\\
\px u+\py v=0,\\
u|_{t=0}=u_0,\\
u|_{y=0}=v|_{y=0}=0,\\
u|_{y\rightarrow\ty}=U.
\end{array}\right.
\end{equation}
Here $(u,v)$ denote the velocity filed. Furthermore $(U,p)$ represent the traces of the tangential field and the pressure of the outflow on the boundary, satisfying Bernoulli's law:
\begin{equation}\label{Bernoulli}
\pt U +U\px U +\px p=0.
\end{equation}
When the power-law index $0<n<1$, the fluid is pseudo-plastic; when $n>1$, it is dilatant; and when $n=1$, it is a Newtonian fluid, in which case the equations \eqref{feiniu} are consistent with the Prandtl equations. And in this paper, we consider the pseudo-plastic case.

Let us first review the research results of the Prandtl equations, which were originally proposed by Prandtl in 1904 to account for the discrepancy between an ideal fluid and a low-viscosity fluid near a boundary. To date, there have been numerous studies on the Prandtl equations. The primary difficulty in studying the Prandtl equations lies in the loss of one tangential derivative. For the 2D case, under the assumption of monotonicity, the local existence and the uniqueness of classical solution were established by Oleinik \cite{OAO} via the Crocco transform. 
Subsequently, Xin and Zhang \cite{ZPX} proved the global existence of the weak solution under an additional favourable pressure condition.
On the other hand, in 2015, \cite{RA} and \cite{NM} independently demonstrated the local well-posedness in weighted Sobolev spaces by using energy methods. 
Together with the ill-posedness result reported in \cite{DG2010}, it is now recognized that the monotonicity assumption is almost a necessary and sufficient condition for well-posedness in Sobolev spaces. In the absence of the monotonicity assumption, the reader is referred to \cite{MI,MCL,MP,MS,PZ} for well-posedness results in analytic spaces and to \cite{DC,HD,DG2015,WXLY2,CW} for those in Gevrey spaces.

For the 3D case, due to the emergence of the secondary flow, the monotonicity assumption alone is insufficient to guarantee the well-posedness of the Prandtl equations in Sobolev space. According to \cite{CJLARMA,CJLADV}, it is known that $\pz(\dfrac{u}{v})\equiv0~\&~\pz u>0$ constitutes an almost necessary and sufficient structural condition for the well-posedness in Sobolev spaces. On the other hand, if there is no structural assumption, the reader is referred to \cite{WXLMY} for the well-posedness result in Gevrey class 2 space.

In addition, the reader is referred to \cite{WXLY,CJLCPAM,ZT,ZEW} for the effect of the magnetic field on the boundary layer and to \cite{ZT2,YGW1,YGW2} for the effect of temperature on the boundary layer.

Compared with Newtonian fluids, there is very little mathematical theory concerning the boundary layer equations for non-Newtonian fluids, and the existing results are mainly focused on numerical ones. In Chapter 8 of \cite{OAO}, Oleinik and Samokhin proved the local existence and uniqueness of classical solutions to the 2D non-Newtonian boundary layer equations by using the Crocco transform. Zhang \cite{JWZ} obtained similar results for another special type of outer flow, also by means of the Crocco transform.

Inspired by the existing literature, in this work we investigate the well-posedness of the boundary layer for a pseudo-plastic fluid by energy methods under Oleinik's monotonicity assumption.

Under the assumption of monotonicity ($\py u>0$), the system \eqref{feiniu} can be simplified to the following
\begin{equation}\label{feiniu2}
\left\{\begin{array}{l}
\pt u +u\px u +v\py u-n(\py u)^{n-1}\py^2 u+\px p=0,\\
\px u+\py v=0,\\
u|_{t=0}=u_0,\\
u|_{y=0}=v|_{y=0}=0,\\
u|_{y\rightarrow+\ty}=U.
\end{array}\right.
\end{equation}
And equivalently, the vorticity $\om:=\py u>0$ satisfies the following vorticity system:
\begin{equation}\label{wodu}
\left\{\begin{array}{l}
\pt \om +u\px \om +v\py \om-n\om^{n-1}\py^2 \om-n(n-1)
\om^{n-2}|\py\om|^2=0,\\
\om|_{t=0}:=\om_0=\py u_0,\\
n\om^{n-1}\py\om|_{y=0}=\px p,\\
\om|_{y\rightarrow+\ty}=0.
\end{array}\right.
\end{equation}

For convenience, we define
$$\Omega:=\{(x,y)|x\in\T,y\in\R_+\}.$$
Then, for any $\gamma\in\R$, use $L_{\gamma}^2(\Omega)$ to denote the weighted Lebesgue space:
\begin{align*}
L_{\gamma}^2:=\left\{f(x,y):\Om\rightarrow\R,~
\|f\|_{L_{\gamma}^2}:=\|(1+y)^{\ga}f\|_{L^2}
<+\ty\right\}
\end{align*}
Accordingly the weighted Sobolev space $H^{s,\ga}$:
\begin{align*}
H^{s,\ga}:=\left\{f(x,y):\Om\rightarrow\R,~
\|f\|_{H^{s,\ga}}^2:=
\sum\limits_{\al_1+\al_2\leq s}\|(1+y)^{\ga+\al_2}\px^{\al_1}\py^{\al_2}f\|_{L^2}^2
<+\ty\right\}
\end{align*}

Now, we can state our main result as follows. 
\begin{Theorem}\label{zhuding}
Assume $\frac{1}{3}<n<1$. Let $s\geq4,~\ga\geq1,~\ga+\frac{1}{2}<\si
\leq\min\{\frac{1}{1-n},\ga+\frac{2}{3}\}$. And we suppose the outer flow $(U,p)$ satisfies
\begin{equation}\label{wailiushangjie}
M:=\sum\limits_{i=0}^{2s+2}
\left(\sup\limits_{t}\|\pt^i
(U,p)(t,\cdot)\|_{H^{3s+2-i}(\T)}\right)<+\ty.
\end{equation}
Assume that the initial velocity $u_0-U|_{t=0}\in H^{3s+2,\ga-1}$ and the initial vorticity $\om_0\in H^{3s+2,\ga}$.
More, there exist constants $c_1$, $c_2>0$ such that for $(x,y)\in\Om$,
\begin{equation}\label{chushiwodu000}
(1+y)^\sigma \omega_0 \geq c_1,
\end{equation}
\begin{equation}\label{chushiwodu111}
\sum_{\al_1+\al_2\leq 4}\left|(1+y)^{\sigma+\alpha_2} \px^{\al_1}\py^{\al_2}\omega_0\right|^2 \leq c_2.
\end{equation}
Then there exist a time $T_*>0$ and a unique solution $(u,v)$ to the pseudo-plastic fluid boundary layer equations \eqref{feiniu} such that
\begin{equation}\label{solution11}
u-U \in \bigcap_{i=0}^s W^{i, \infty}\left(0, T_* ; H^{s-i,\ga-1}(\Omega)\right),
\end{equation}
\begin{equation}\label{solution22}
v+U_x y \in \bigcap_{i=0}^{s-1} W^{i, \infty}\left(0, T_* ; H^{s-1-i,-1}(\Omega)\right),
\end{equation}
\begin{equation}\label{solution33}
\py v+U_x \in \bigcap_{i=0}^{s-1} W^{i, \infty}\left(0, T_* ; H^{s-1-i,\ga-1}(\Omega)\right),
\end{equation}
and the vorticity 
\begin{equation}\label{solution44}
\om \in \bigcap_{i=0}^s W^{i, \infty}\left(0, T_* ; H^{s-i,\ga}(\Omega)\right).
\end{equation}
In addition, if $\ga>\frac{3}{2}$,
\begin{equation}\label{solution55}
v+U_x y \in \bigcap_{i=0}^{s-1} W^{i, \infty}\left(0, T_* ; L^{\ty}\left(\R_{y,+};H^{s-1-i}(\T_{x})\right)\right).
\end{equation}
\end{Theorem}
\begin{Remark}
We require $\frac{1}{3}<n<1$ because we need to ensure $\frac{1}{1-n}>\frac{3}{2}$; otherwise, there would be no $\ga$ and $\si$ satisfying $\ga\geq1$ and $\ga+\frac{1}{2}<\si
\leq\min\{\frac{1}{1-n},\ga+\frac{2}{3}\}$.
\end{Remark}
\begin{Remark}
It should be noted that the assumption on the outflow $(U,p)$ and the initial data $(u_0,\om_0)$ is not optimal. Here, we need this regularity to simplify the regularization procedure when constructing approximate solutions.
\end{Remark}
\begin{Remark}
\eqref{chushiwodu000} and \eqref{chushiwodu111} are used to ensure that \eqref{chushiwodu222} holds.
\end{Remark}

The key step in proving Theorem \ref{zhuding} is to obtain a priori estimates for the vorticity system \eqref{wodu}; see Proposition \ref{xianyanguji} for details. Based on Proposition \ref{xianyanguji}, a standard regularization procedure yields the local-in-time existence and uniqueness. Therefore, we only prove Proposition \ref{xianyanguji} in Section \ref{zhumingti} and omit the regularization procedure for brevity. The reader may refer to \cite[Sec. 4]{CJLCPAM} or \cite{NM} for a detailed account of the regularization procedure.
\section{Preliminaries}\label{zhunbei}
Let's first introduce some notation. Denote derivative (in both space and time) operator by
$$D^{\al}=\pt^{\al_1}\px^{\al_2}\py^{\al_3}
~~\text{for}~~\al=(\al_1,\al_2,\al_3)\in\N^3,~~
|\al|=\al_1+\al_2+\al_3.$$
For ease of expression, when we want to emphasize $\al_{3}=0$, we will use $\ptau^{\al}$ to denote $D^{\al}$, that is
$$\ptau^{\al}=\pt^{\al_1}\px^{\al_2}
~~\text{for}~~\al=(\al_1,\al_2,0)\in\N^3.$$
And let $e_i\in\N^3,~i=1,2,3$ be the following
$$e_1=(1,0,0),~e_2=(0,1,0),~e_3=(0,0,1)\in\N^3.$$
Then it is obvious that
\begin{align*}
\pt=D^{e_1},
~\px=D^{e_2},~
\py=D^{e_3}.
\end{align*}

Use $\py^{-1}$ to denote the inverse of the derivative $\py$, that is $\py^{-1}f(y)=\int_0^y f(\tilde{y})d\tilde{y}.$ 
Use $\mathcal{P}(\cdot)$
to denote a nondecreasing polynomial function, which may differ from line to line. Use $C$ to denote a non-negative constant that may vary from line to line. Use $[a]$ to denote the greatest integer not exceeding $a$.

In order to solve the pseudo-plastic fluid boundary layer equations \eqref{feiniu} in Sobolev spaces, we need to overcome the following two main difficulties:
\begin{itemize}
  \item Derivative loss. From the divergence-free condition and the boundary condition $\eqref{feiniu}_4$, we know that $v=-\py^{-1}\px u$, which creates a loss of the $x$-derivative, hence we can't use the standard energy estimates.
  \item The difficulty caused by viscosity. Specifically, $\om^{n-1}$ and its higher-order derivatives diverge as $y\rightarrow+\ty$.
\end{itemize}

The core idea to overcome the above two difficulties is to use the monotonicity assumption. To this end, we first define the function space $\mH^{s,\gamma}_{\si,\de}$ for $\om$ by
$$
\begin{aligned}
\mH_{\sigma, \delta}^{s, \gamma}:= & \bigg\{\omega: \Om \rightarrow \mathbb{R}:\|\omega\|_{\mH^{s, \gamma}}<+\infty,(1+y)^\sigma \omega \geq \delta, \\
& \text { and } \sum_{|\alpha| \leq 2}\left|(1+y)^{\sigma+\alpha_3} D^\alpha \omega\right|^2 \leq \frac{1}{\delta^2}\bigg\},
\end{aligned}
$$
where $s\geq4,~\ga\geq1,~\si>\ga+\frac{1}{2},~\de\in(0,1)$. Apart from that, the weighted norm $\|\cdot\|_{\mH^{s, \gamma}}$ denotes 
\begin{equation}\label{hsga}
\|\omega\|_{\mH^{s, \gamma}}^2:=\sum_{|\alpha| \leq s}\left\|(1+y)^{\gamma+\alpha_3} D^\alpha \omega\right\|_{L^2\left(\Om\right)}^2.
\end{equation}
Correspondingly, we also denote
$$
\mH^{s, \gamma}:=\left\{\omega: \Om \rightarrow \mathbb{R}:\|\omega\|_{\mH^{s, \gamma}}<+\infty\right\}.
$$
\begin{Remark}
In the definition of $\mH^{s,\gamma}_{\si,\de}$, we specify that $\si>\ga+\frac{1}{2}$; if $\si\leq\ga+\frac{1}{2}$, then $\mH^{s,\gamma}_{\si,\de}$ would be an empty set.
\end{Remark}
\begin{Remark}
Note that our function space is slightly different from that in \cite{NM}. In our paper, $\mH^{s,\gamma}_{\si,\de}$ includes time derivatives, whereas \cite{NM} does not. This is because we wish to remove the restriction that $s$ is even, and therefore we adopt a different treatment of the boundary terms (see \eqref{K140}) from that in \cite{NM}. However, to facilitate the reader's comparison of the results obtained in the two function spaces, we will present the results in an alternative function space in Section \ref{lingyizhong}.
\end{Remark}

For the first difficulty, we refer to the cancellation mechanism proposed in \cite{NM} and define the following norm:
$$
\|\om\|_{\mH^{s, \gamma}_g}^2=
\sum_{\substack{|\al|\leq s \\ \al_1+\al_2\leq s-1}}
\|D^{\al}\om\|^2_{L^2_{\ga+\al_3}}
+\sum_{\al_1+\al_2=s}
\|g_{\al}\|^2_{L^2_{\ga}},
$$
where $g_{\al}:=\ptau^{\al}\om-a\ptau^{\al}(u-U)$ and $a:=\frac{\py\om}{\om}$. 

On the one hand, $g_{\al}$ can avoid the loss of $x$-derivative; see Section \ref{zuoqiexiangguji} for details.
On the other hand, we know from \cite[Appendix A]{NM} that there exist positive constants $C_1$ and $C_2$ such that
\begin{equation}\label{chabuduo}
\begin{split}
C_1\|\om\|_{\mH^{s, \gamma}_g}
\leq \|u-U\|_{\mH^{s, \ga-1}}
+\|\omega\|_{\mH^{s, \gamma}}
\leq C_2(\|\om\|_{\mH^{s, \gamma}_g}+\sum_{\al_1+\al_2=s}
\|\ptau^{\al} U\|_{L^2_{x}(\T)}).
\end{split}
\end{equation}
Therefore, in order to overcome the first difficulty, we will estimate $\|\om\|_{\mH^{s, \gamma}_g}$ instead of $\|\omega\|_{\mH^{s, \gamma}}$.

For the second difficulty, we need to precisely characterize the decay rates of $\om^{n-1}$ and its higher-order derivatives as $y\rightarrow+\ty$.

To handle higher-order derivatives of $\om^{n-1}$, we need the following Fa\`a di Bruno formula; see \cite{JA}.
\begin{Lemma}[Fa\`a di Bruno formula]\label{faa}
For a multi-index $\al=(\al_1,\al_2,\al_3)$ and a real number $m$, let $f(t,x,y)>0$ be sufficiently smooth. Then there holds
\begin{equation}\label{faa1}
D^\alpha (f^m) = \sum_{r=1}^{|\al|} m^{\underline{r}} \, f^{m-r} \sum_{\substack{\beta^1 + \cdots + \beta^r = \alpha \\ |\beta^i| \ge 1}} \frac{\alpha!}{r!} \prod_{i=1}^r \frac{D^{\beta^i} f}{\beta^i!},
\end{equation}
where:
\begin{itemize}
  \item $m^{\underline{r}}=m(m-1)\cdots(m-r+1)$ is the falling factorial (this term vanishes when $r>m$ if $m$ is an positive integer.)
  \item $\al!=\al_1!\al_2!\al_3!$, and similarly for $\beta^i!$.
  \item The inner sum runs over all ordered $r$-tuples $(\beta^1,\cdots,\beta^r)$ of non-zero multi-indices whose sum equals $\al$.
\end{itemize}
\end{Lemma}

Moreover, for $\om\in\mH^{s,\gamma}_{\si,\de}$, we have the following decay property of $D^{\al}\om$ as $y$ goes to $+\ty$, which is similar to \cite[Remark C.4]{NM}.
\begin{Proposition}\label{shuaijian}
Let $s\geq5$ be an integer, $\ga\geq1,~\si>\ga+\frac{1}{2}$ and $\de\in(0,1)$. If $\om\in\mH^{s,\gamma}_{\si,\de}$, there exists a constant $C>0$, which depends on $s,~\ga,~\si$ and $\de$, such that for all $|\al|\leq s-2$,
\begin{equation}\label{shuaijian1}
|D^{\al}\om|\leq Cc_\alpha(1+y)^{-b_{\al}}\quad\text{in}~\T\times\R_+
\end{equation}
where 
\begin{equation}\label{shuaijian2}
b_\alpha := 
\begin{cases} 
\sigma + \alpha_3 & \text{if } |\alpha| \leq 2, \\ 
\frac{(s-2 - |\alpha|)\sigma + (|\alpha| - 2)\gamma}{s-4} + \alpha_3 & \text{if } 3 \leq |\alpha| \leq s-3, \\ 
\gamma + \alpha_3 & \text{if } |\alpha| = s-2,
\end{cases}
\end{equation}
and 
\begin{equation}\label{shuaijian3}
c_\alpha := 
\begin{cases} 
1 & \text{if } |\alpha| \leq 2, \\ 
1+\|\omega\|_{\mH^{s, \gamma}} & \text{if } 3 \leq |\alpha| \leq s-3, \\ 
\|\omega\|_{\mH^{s, \gamma}} & \text{if } |\alpha| = s-2.
\end{cases}
\end{equation}
\end{Proposition}
The derivation of the coefficient $c_\alpha$ utilizes \cite[Lemma C.3]{NM} and the fact $2\sqrt{C_0C_2}\leq C_0+C_2$. And we omit the proof for brevity.
\begin{Remark}\label{shuaijian4}
To reduce the number of cases, we can also directly write it in the following form.
\begin{equation}\label{shuaijian5}
|D^{\al}\om|\leq C(1+\|\omega\|_{\mH^{s, \gamma}})(1+y)^{-b_{\al}}\quad\text{in}~\T\times\R_+
\end{equation}
where
\begin{equation}\label{shuaijian6}
b_\alpha := 
\begin{cases} 
\sigma + \alpha_3 & \text{if } |\alpha| \leq 1, \\ 
\frac{(s-2 - |\alpha|)\sigma + (|\alpha| - 2)\gamma}{s-4} + \alpha_3 & \text{if } 2 \leq |\alpha| \leq s-2.
\end{cases}
\end{equation}
\end{Remark}

With Lemma \ref{faa} and Proposition \ref{shuaijian}, we are able to overcome the second difficulty.

Next, for the estimation of the $L^{\ty}$ norm of $D^{\al}\om,~|\al|\leq2$, we need the following two extremum principles. Their proofs are almost identical to those of \cite[Lemma E.1]{NM} and \cite[Lemma E.2]{NM}, so we omit them here.
\begin{Lemma}[Maximum Principle for Parabolic Equations]
\label{jida}
Let $\epsilon \geq 0$. If $H \in C([0, T]; C^2(\mathbb{T} \times \mathbb{R}_+)) \cap C^1([0, T]; C^0(\mathbb{T} \times \mathbb{R}_+))$ is a bounded function that satisfies the following
\[
\left\{ \partial_t + b_1 \partial_x + b_2 \partial_y - \epsilon^2 \partial_x^2 - b_3\partial_y^2 \right\} H \leq fH \quad \text{in } [0, T] \times \mathbb{T} \times \mathbb{R}_+,
\]
where the coefficients $b_1, b_2$, $b_3$ and $f$ are continuous and satisfy
\begin{equation}\label{eq:E.1}
b_3>0,\quad
\left\| \frac{b_2}{1 + y} \right\|_{L^\infty([0, T] \times \mathbb{T} \times \mathbb{R}_+)} < +\infty \quad \text{and} \quad \|f\|_{L^\infty([0, T] \times \mathbb{T} \times \mathbb{R}_+)} \leq \lambda,
\end{equation}
then for any $t \in [0, T]$,
\begin{equation}\label{eq:E.2}
\sup_{\mathbb{T} \times \mathbb{R}_+} H(t) \leq
\max \left\{ e^{\lambda t} \|H(0)\|_{L^\infty(\mathbb{T} \times \mathbb{R}_+)}, 
\max_{\tau \in [0, t]} \left\{ e^{\lambda(t-\tau)} \|H(\tau)|_{y=0}\|_{L^\infty(\mathbb{T})} \right\} \right\}.
\end{equation}
\end{Lemma}
\begin{Lemma}[Minimum Principle for Parabolic Equations]
\label{jixiao}
Let $\epsilon \geq 0$. If $H \in C([0, T]; C^2(\mathbb{T} \times \mathbb{R}_+)) \cap C^1([0, T]; C^0(\mathbb{T} \times \mathbb{R}_+))$ is a bounded function with
\[
\kappa(t) := \min \left\{ \min_{\mathbb{T} \times \mathbb{R}_+} H(0), \min_{[0,t] \times \mathbb{T}} H\big|_{y=0} \right\} \geq 0
\]
and satisfies
\[
\left\{ \partial_t + b_1 \partial_x + b_2 \partial_y - \epsilon^2 \partial_x^2 - b_3\partial_y^2 \right\} H = fH
\]
where the coefficients $b_1, b_2$, $b_3$ and $f$ are continuous and satisfy \eqref{eq:E.1}, then for any $t \in [0, T]$,
\begin{equation}\label{eq:E.3}
\min_{\mathbb{T} \times \mathbb{R}_+} H(t) \geq (1 - \lambda t e^{\lambda t}) \kappa(t).
\end{equation}
\end{Lemma}

Now, we present some commonly used inequalities, which can be found in \cite{CJLCPAM}, see also \cite{NM,CJX}.
\begin{Lemma}\label{jingchangyong}
For proper functions $f_1,f_2$, there holds:
\begin{description}
  \item[(i)] If $\lim_{y\rightarrow+\ty}(f_1f_2)(x,y)=0$, then we have
\begin{equation}\label{L1}
\left|\int_{\T}(f_1f_2)|_{y=0}dx\right|
\leq \|\py f_1\|_{L^2(\Om)}\|f_2\|_{L^2(\Om)}
+\|f_1\|_{L^2(\Om)}\|\py f_2\|_{L^2(\Om)}.
\end{equation}
In particular, if $\lim_{y\rightarrow+\ty}f_1(x,y)=0$, then
\begin{equation}\label{L1.5}
\left\|f_1|_{y=0}\right\|_{L^2(\T)}
\leq \sqrt{2}\|f_1\|^{\frac{1}{2}}_{L^2(\Om)}\|\py f_1\|^{\frac{1}{2}}_{L^2(\Om)}.
\end{equation}
  \item[(ii)] Let $\ga\in\R$ and an integer $s\geq3$, then for any $\al=(\al_1,\al_2,\al_3)\in \N^3$ and $\tilde{\al}=(\tilde{\al}_1,\tilde{\al}_2,\tilde{\al}_3)\in \N^3$, which satisfy $|\al|+|\tilde{\al}|\leq s$, there holds
\begin{equation}\label{L2}
\|(D^{\al}f_1\cdot D^{\tilde{\al}}f_2)(t,\cdot)\|_{L^2_{\gamma+\al_3+\tilde{\al}_3}(\Om)}
\leq C\|f_1(t)\|_{\mH^{s, \gamma_1}}\|f_2(t)\|_{\mH^{s, \gamma_2}},
\end{equation}
where $\ga_1,\ga_2\in\R$ and $\ga_1+\ga_2=\ga$.
  \item[(iii)] For any $\la>\frac{1}{2}$, $\tilde{\la}>0$, we have
\begin{equation}\label{L3}
\begin{aligned}
& \|(1+y)^{-\lambda}(\partial_y^{-1} f_1)(y)\|_{L_y^2(\mathbb{R}_{+})} \leq \frac{2}{2 \lambda-1}\|(1+y)^{1-\lambda} f_1(y)\|_{L_y^2(\mathbb{R}_{+})}, \\
& \|(1+y)^{-\tilde{\lambda}}(\partial_y^{-1} f_1)(y)\|_{L_y^{\infty}(\mathbb{R}_{+})} \leq \frac{1}{\tilde{\lambda}}\|(1+y)^{1-\tilde{\lambda}} f_1(y)\|_{L_y^{\infty}(\mathbb{R}_{+})}.
\end{aligned}
\end{equation}
And let $\gamma\in\R$ and an integer $s\geq3$, then for any $\al=(\al_1,\al_2,\al_3)\in \N^3$ and $\tilde{\beta}=(\tilde{\beta}_1,\tilde{\beta}_2,0)\in \N^3$, which satisfy $|\al|+|\tilde{\beta}|\leq s$, there holds
\begin{equation}\label{L4}
\|(D^{\al}f_1\cdot \ptau^{\tilde{\beta}}
\py^{-1}f_2)(t,\cdot)\|_{L^2_{\gamma+\al_3}(\Om)}\leq
C \|f_1(t)\|_{\mH^{s,\gamma+\la}}\|f_2(t)\|_{\mH^{s,1-\la}}.
\end{equation}
In particular, for $\la=1$
\begin{equation}\label{L4.5}
\begin{aligned}
 \|(1+y)^{-1}(\partial_y^{-1} f_1)(y)\|_{L_y^2(\mathbb{R}_{+})}
&\leq 2\|f_1(y)\|_{L_y^2(\mathbb{R}_{+})}, \\
\|(D^{\al}f_1\cdot \ptau^{\tilde{\beta}}
\py^{-1}f_2)(t,\cdot)\|_{L^2_{\gamma+\al_3}(\Om)}
&\leq
C \|f_1(t)\|_{\mH^{s,1+\gamma}}\|f_2(t)\|_{\mH^{s,0}}.
\end{aligned}
\end{equation}
\item[(iv)]For any $\la>\frac{1}{2}$, we have
\begin{equation}\label{L5}
\|(\py^{-1}f_1)(y)\|_{L_y^{\ty}(\R_+)}
\leq C\|f_1\|_{L^2_{y,\la}(\R_+)}.
\end{equation}
And let $\gamma\in\R$ and an integer $s\geq2$, then for any $\al=(\al_1,\al_2,\al_3)\in \N^3$ and $\tilde{\beta}=(\tilde{\beta}_1,\tilde{\beta}_2,0)\in \N^3$, which satisfy $|\al|+|\tilde{\beta}|\leq s$, there holds
\begin{equation}\label{L6}
\|(D^{\al}f_1\cdot \ptau^{\tilde{\beta}}
\py^{-1}f_2)(t,\cdot)\|_{L^2_{\gamma+\al_3}(\Om)}\leq
C \|f_1(t)\|_{\mH^{s,\gamma}}\|f_2(t)\|_{\mH^{s,\la}}.
\end{equation}
\end{description}
\end{Lemma}
Moreover, we have the following Sobolev-type inequality, which is \cite[Lemma B.2]{NM}
\begin{Lemma}\label{chazhi}
Let $\phi~:~\T\times\R_+\rightarrow \R$. Then there exists a universal constant $\tilde{C}>0$ such that
\begin{equation}\label{chazhi1}
\|\phi\|_{L^{\ty}}\leq \tilde{C} (\|\phi\|_{L^2}
+\|\px\phi\|_{L^2}+\|\py^2\phi\|_{L^2}).
\end{equation}
\end{Lemma}

\section{A Priori Estimates}\label{zhumingti}
In this section, we establish the a priori estimates for system \eqref{wodu}, which are required for the proof of Theorem \ref{zhuding}.
\begin{Proposition}\label{xianyanguji}
Assume $\frac{1}{3}<n<1$. Let $s\geq4,~\ga\geq1,~\ga+\frac{1}{2}<\si
\leq\min\{\frac{1}{1-n},\ga+\frac{2}{3}\}$ and
$\de\in(0,1)$ is small enough.
Assume all the hypotheses for $(U,p)$ and $(u_0,\om_0)$ given in Theorem \ref{zhuding} hold.
If $(u,v,\om)$ is a classical solution of \eqref{wodu} in $[0,T]$ and satisfies
\begin{equation}\label{jiashe}
\om\in L^{\ty}(0,T;\mH^{s,\gamma}_{\si,\de}),~\py\om\in L^{2}(0,T;\mH^{s,\gamma}).
\end{equation}
Then, there exists a positive constant $C$, depending only on $s,~\ga,~\si,~\de$ and $n$, such that for small time,
\begin{equation}\label{zuihou1}
\begin{split}
\|\om\|_{\mH^{s, \gamma}_g}^2
\leq Q(t), 
\end{split}
\end{equation}
where
\begin{equation}\label{qt}
\begin{split}
Q(t):=&\{\mathcal{P}(M,\|u_0-U|_{t=0}\|_{H^{2s,\ga-1}},\|\om_0\|_{H^{2s,\ga}})+
C(1+M)^{2s+2}t \}\\
&\cdot\left\{
1-Cs\{\mathcal{P}(M,\|u_0-U|_{t=0}\|_{H^{2s,\ga-1}},\|\om_0\|_{H^{2s,\ga}})+
C(1+M)^{2s+2}t\}^s t
\right\}^{-\frac{1}{s}}
\end{split}
\end{equation}
Also, we have that
\begin{equation}\label{zuihou2}
\begin{split}
(1+y)^{\si}\om(\tau)\geq (1+y)^{\si}\om_0-\de^{-1}t,
\end{split}
\end{equation}
and
\begin{equation}\label{zuihou3}
\begin{split}
\|I(t)\|_{L^{\ty}}
\leq \max\left\{
\|I(0)\|_{L^{\ty}},~
6\tilde{C}^2 A^2(t)\right\}e^{C (1+M+A(t))t},
\end{split}
\end{equation}
where $I(t):=\sum_{|\alpha| \leq 2}\left|(1+y)^{\sigma+\alpha_3} D^\alpha \omega\right|^2$ and the universal constant $\tilde{C}$ is the same as the in Lemma \ref{chazhi}. Moreover, $A(t)$ is defined by
\begin{equation}\label{At}
\begin{split}
A(t):=\max_{[0,t]}\|\om\|_{\mH^{s, \gamma}_g}.
\end{split}
\end{equation}
In addition, if $s\geq5$, then there also holds
\begin{equation}\label{zuihou4}
\begin{split}
\|I(t)\|_{L^{\ty}}
\leq \left\{
\|I(0)\|_{L^{\ty}}+4\tilde{C}^2 A^2(t)t
\right\}e^{C (1+M+A(t))t}
.
\end{split}
\end{equation}

\end{Proposition}
\begin{Remark}\label{bixuyou}
By using \eqref{wodu}, we know that $I(0)$ can be expressed by spatial derivatives of initial data $(\om_0,u_0)$ up to order $4$. Then by \eqref{chushiwodu000} and \eqref{chushiwodu111}, we know that for $\de$ small enough (depending on $\|u_0\|_{W^{4,\ty}}$), there holds
\begin{equation}\label{chushiwodu222}
(1+y)^\sigma \omega_0 \geq 2\de,\quad \|I(0)\|_{L^{\ty}} \leq \frac{1}{4\de^2}.
\end{equation}
\end{Remark}

We will divide the proof of Proposition \ref{xianyanguji} into two parts. First, for $\al$ satisfying $|\al|\leq s$ and $\al_1+\al_2\leq s-1$, the a priori estimates can be obtained by the standard energy method because there is no tangential regularity loss. In the second case that $\al_1+\al_2=s$ and $\al_3=0$, we will use $g_{\al}$ to cancel the tangential regularity loss.

\subsection{Estimates with Normal Derivatives}\label{zuofaxiangguji}
First, we will deal with the weighted estimates for $D^{\al}\om$ with $|\al|\leq s$ and $\al_1+\al_2\leq s-1$. The result reads as the following:
\begin{Proposition}[Weighted estimates for $D^{\al}\om$ with $|\al|\leq s$ and $\al_1+\al_2\leq s-1$]\label{guji1}
Assume $\frac{1}{3}<n<1$. Let $s\geq4,~\ga\geq1,~\ga+\frac{1}{2}<\si
\leq\min\{\frac{1}{1-n},\ga+\frac{2}{3}\}$ and
$\de\in(0,1)$ is small enough. If $(u,v,\om)$ is a classical solution of \eqref{wodu} in $[0,T]$ and satisfies
\begin{equation*}
\om\in L^{\ty}(0,T;\mH^{s,\gamma}_{\si,\de}),~\py\om\in L^{2}(0,T;\mH^{s,\gamma}),
\end{equation*}
then there exists a positive constant C, which depends on $n,~s,~\ga,~\si$ and $\de$ such that for any small $0<\ep<1$,
\begin{equation}\label{qieguji}
\begin{split}
&\sum_{\substack{|\al|\leq s \\ \al_1+\al_2\leq s-1}}
\left(\frac{d}{dt}\|D^{\al}\om\|^2_{L^2_{\ga+\al_3}}
+ n\de^{1-n}\|\py D^{\al}\om\|^2_{L^2_{\ga+\al_3}}
\right)\\
\leq &C\ep\|\py\omega\|_{\mH^{s, \gamma}}^2
+C\ep^{-1}(1+M+\|u-U\|_{\mH^{s, \ga-1}}+\|\omega\|_{\mH^{s, \gamma}})^{2s+2}.
\end{split}
\end{equation}

\end{Proposition}

\textbf{Proof of Proposition \ref{guji1}.}
Let $\al$ satisfy $|\al|\leq s$ and $\al_1+\al_2\leq s-1$. Applying $D^{\al}$ to \eqref{wodu}, it yields
\begin{equation}\label{alwodu}
\begin{split}
&\pt D^{\al}\om +u\px D^{\al}\om +v\py D^{\al}\om-n\om^{n-1}\py^2 D^{\al}\om\\
=
&-\sum_{0<\beta\leq\al}\binom{\al}{\beta}
\left(D^{\beta}u\px D^{\al-\beta}\om+D^{\beta}v\py D^{\al-\beta}\om\right)\\
&+n\sum_{0<\beta\leq\al}\binom{\al}{\beta}
D^{\beta}\om^{n-1}\py^2 D^{\al-\beta}\om
+n(n-1)\sum_{0\leq\beta\leq\al}\binom{\al}{\beta}
D^{\beta}\om^{n-2}D^{\al-\beta}|\py\om|^2.
\end{split}
\end{equation}
Taking $L^2$ inner product of \eqref{alwodu} with $(1+y)^{2\ga+2\al_3}D^{\al}\om$ yields
\begin{equation}\label{neiji}
\begin{split}
&\frac{1}{2}\frac{d}{dt}\|(1+y)^{\ga+\al_3} D^{\al}\om\|^2_{L^2}\\
=&\int_{\Om}n(1+y)^{2\ga+2\al_3}\om^{n-1}D^{\al}\om\py^2 D^{\al}\om dxdy\\
&-\int_{\Om}(1+y)^{2\ga+2\al_3}D^{\al}\om
\Big(u\px D^{\al}\om +v\py D^{\al}\om\Big) dxdy\\
&-\sum_{0<\beta\leq\al}\binom{\al}{\beta}
\int_{\Om}(1+y)^{2\ga+2\al_3}D^{\al}\om\Big(D^{\beta}u\px D^{\al-\beta}\om+D^{\beta}v\py D^{\al-\beta}\om\Big)dxdy\\
&+n\sum_{0<\beta\leq\al}\binom{\al}{\beta}
\int_{\Om}(1+y)^{2\ga+2\al_3}D^{\al}\om
D^{\beta}\om^{n-1}\py^2 D^{\al-\beta}\om dxdy\\
&+n(n-1)\sum_{0\leq\beta\leq\al}\binom{\al}{\beta}
\int_{\Om}(1+y)^{2\ga+2\al_3}D^{\al}\om
D^{\beta}\om^{n-2}D^{\al-\beta}|\py\om|^2 dxdy\\
:=&\sum\limits_{i=1}^5 K_i.
\end{split}
\end{equation}

For $K_1$, by integration by parts, we have
\begin{equation}\label{K1jisuan}
\begin{split}
K_1=
&-n\int_{\Omega}(1+y)^{2\ga+2\al_3}\om^{n-1}|\py D^{\al}\om|^2 dxdy\\
&-n(n-1)\int_{\Omega}(1+y)^{2\ga+2\al_3}\om^{n-2}\py\om D^{\al}\om
\py D^{\al}\om dxdy\\
&-n(2\ga+2\al_3)\int_{\Omega}(1+y)^{2\ga+2\al_3-1}\om^{n-1} D^{\al}\om
\py D^{\al}\om dxdy\\
&-n\int_{\T}(\om^{n-1} D^{\al}\om
\py D^{\al}\om)|_{y=0}dx\\
:=&\sum\limits_{i=1}^4 K^i_{1}.
\end{split}
\end{equation}
For $K^1_{1}$, since $\om\in\mH^{s,\gamma}_{\si,\de}$, we have 
\begin{equation}\label{K111}
\de\leq(1+y)^{\si}\om\leq\de^{-1}.
\end{equation}
On the other hand, from $\frac{1}{3}<n<1$, we know
\begin{equation}\label{K112}
(1+y)^{\si(n-1)}\om^{n-1}\geq\de^{1-n},
\end{equation}
then
\begin{equation}\label{K113}
\om^{n-1}\geq(1+y)^{\si(1-n)}\de^{1-n}\geq \de^{1-n}.
\end{equation}
Applying \eqref{K113}, we obtain
\begin{equation}\label{K114}
\begin{split}
K^1_1&\leq-n\de^{1-n}\int_{\Om}(1+y)^{2\ga+2\al_3}|\py D^{\al}\om|^2 dxdy\\
&=-n\de^{1-n}\|(1+y)^{\ga+\al_3}\py D^{\al}\om\|^2_{L^2}.
\end{split}
\end{equation}
For $K^2_{1}$, similarly, with $\om\in\mH^{s,\gamma}_{\si,\de}$, we can obtain 
\begin{equation}\label{K121}
|(1+y)^{\si(n-l)}\om^{n-l}|\leq\de^{n-l},
\end{equation}
for any $l\geq1$, and 
\begin{equation}\label{K122}
|(1+y)^{\si+1}\py\om|\leq\de^{-1}.
\end{equation}
Using \eqref{K121} and \eqref{K122} yields 
\begin{equation}\label{K123}
|\om^{n-2}\py\om|
\leq(1+y)^{-\si(n-1)-1}\de^{n-3}.
\end{equation}
Noting that $\si\leq\frac{1}{1-n}$, we have 
\begin{equation}\label{henyouyong}
-\si(n-1)-1\leq0.
\end{equation}
Substituting it into \eqref{K123}, we deduce that 
\begin{equation}\label{K124}
\begin{split}
|\om^{n-2}\py\om|
\leq\de^{n-3}.
\end{split}
\end{equation}
Substituting \eqref{K124} into $K^2_{1}$, we can get
\begin{equation}\label{K125}
\begin{split}
|K^2_1|&\leq n(1-n)\de^{n-3}\int_{\Omega}(1+y)^{2\ga+2\al_3}|D^{\al}\om
\py D^{\al}\om| dxdy\\
&\leq n(1-n)\de^{n-3} 
\|(1+y)^{\ga+\al_3}D^{\al}\om\|_{L^2}
\|(1+y)^{\ga+\al_3}\py D^{\al}\om\|_{L^2}\\
&\leq\ep\|\py\omega\|_{\mH^{s, \gamma}}^2
+C\ep^{-1}\|\omega\|_{\mH^{s, \gamma}}^2,
\end{split}
\end{equation}
for any small $0<\ep<1$.

For $K^3_{1}$, by using \eqref{K121} and \eqref{henyouyong}, we have the following formula holds,
\begin{equation}\label{K131}
\begin{split}
|K^3_1|
\leq &n(2\ga+2\al_3)\int_{\Omega}
|(1+y)^{-1}\om^{n-1}|
(1+y)^{2\ga+2\al_3}|D^{\al}\om
\py D^{\al}\om| dxdy\\
\leq &n(2\ga+2\al_3)\de^{n-1} 
\|(1+y)^{\ga+\al_3}D^{\al}\om\|_{L^2}
\|(1+y)^{\ga+\al_3}\py D^{\al}\om\|_{L^2}\\
\leq&\ep\|\py\omega\|_{\mH^{s, \gamma}}^2
+C\ep^{-1}\|\omega\|_{\mH^{s, \gamma}}^2.
\end{split}
\end{equation}
Finally, we turn to address the boundary $K^4_1$. We have the following estimate, which will be shown later.
\begin{equation}\label{K140}
\begin{split}
|K^4_1|\leq C\ep\|\py\omega\|_{\mH^{s, \gamma}}^2
+C\ep^{-1}(1+M+\|u-U\|_{\mH^{s, \ga-1}}+\|\omega\|_{\mH^{s, \gamma}})^{2s+2}.
\end{split}
\end{equation}
After substituting the estimates for all $K^i_1$ into $K_1$, we can obtain
\begin{equation}\label{K1deguji}
\begin{split}
K_1\leq -n\de^{1-n}\|(1+y)^{\ga+\al_3}&\py D^{\al}\om\|^2_{L^2}+C\ep\|\py\omega\|_{\mH^{s, \gamma}}^2\\
&+C\ep^{-1}(1+M+\|u-U\|_{\mH^{s, \ga-1}}+\|\omega\|_{\mH^{s, \gamma}})^{2s+2}.
\end{split}
\end{equation}

For $K_2$, by integration by parts, we have
\begin{equation*}
\begin{split}
|K_2|=(\ga+\al_3)\Big|\int_{\Om}(1+y)^{2\ga+2\al_3-1}v(D^{\al}\om)^2 dxdy\Big|\\
\leq (\ga+\al_3)\|(1+y)^{-1}v\|_{L^{\ty}}
\|D^{\al}\om\|_{L^2}^2.
\end{split}
\end{equation*}
On the other hand, by $\eqref{L3}_2$ and Sobolev embedding inequality, we obtain
\begin{equation}\label{vdeguji}
\begin{split}
\|(1+y)^{-1}v\|_{L^{\ty}}=
&\|(1+y)^{-1}\py^{-1}D^{e_2}u\|_{L^{\ty}}\\
\leq&\|D^{e_2}u\|_{L^{\ty}}\\
\leq&\|D^{e_2}(u-U)\|_{L^{\ty}}+\|D^{e_2}U\|_{L^{\ty}_x}\\
\leq&\|D^{e_2}(u-U)\|_{H^{2}}+\|D^{e_2}U\|_{H^{1}_x}\\
\leq&\|u-U\|_{\mH^{s, \gamma-1}}+M.
\end{split}
\end{equation}
Hence, we can estimate $K_2$ according to the following formula.
\begin{equation}\label{K2deguji}
\begin{split}
|K_2|\leq C(\|u-U\|_{\mH^{s, \gamma-1}}+M)\|\omega\|_{\mH^{s, \gamma}}^2.
\end{split}
\end{equation}

For $K_3$, $K_4$ and $K_5$, we have the following estimates, which will be shown later.
\begin{equation}\label{K3deguji}
\begin{split}
|K_3|\leq C (\|u-U\|_{\mH^{s, \ga-1}}+\|\om\|_{\mH^{s, \ga}}+M)
\|\om\|_{\mH^{s, \ga}}^2,
\end{split}
\end{equation}
\begin{equation}\label{K4deguji}
\begin{split}
|K_4|\leq \ep\|\py\omega\|_{\mH^{s, \gamma}}^2
+C\ep^{-1}(1+\|\omega\|_{\mH^{s, \gamma}})^{s},
\end{split}
\end{equation}
and
\begin{equation}\label{K5deguji}
\begin{split}
|K_5|\leq \ep\|\py\omega\|_{\mH^{s, \gamma}}^2
+C\ep^{-1}(1+\|\omega\|_{\mH^{s, \gamma}})^{s+1}.
\end{split}
\end{equation}

Substituting all estimates of $K_i$ into \eqref{neiji} and summing over $\al$ , we can establish that \eqref{qieguji} holds.

\hfill $\square$

$Proof~of~\eqref{K3deguji}$: First, based on H\"{o}lder's inequality, it yields that
\begin{equation}\label{K3jisuan}
\begin{split}
&\Big|\int_{\Om}(1+y)^{2\ga+2\al_3}D^{\al}\om D^{\beta}u\px D^{\al-\beta}\om dxdy\Big|\\
\leq
&\|(1+y)^{\ga+\al_3}D^{\beta}u\px D^{\al-\beta}\om\|_{L^2}
\|(1+y)^{\ga+\al_3}D^{\al}\om\|_{L^2}.
\end{split}
\end{equation}
If $\beta_3>0$, then using \eqref{L2} yields
\begin{equation*}
\begin{split}
\|(1+y)^{\ga+\al_3}D^{\beta}u\px D^{\al-\beta}\om\|_{L^2}
=&\|(1+y)^{\ga+\al_3}D^{\beta-e_3}\om D^{\al-\beta+e_2}\om\|_{L^2}\\
\leq &C\|\om\|_{\mH^{s, \ga}}^2.
\end{split}
\end{equation*}
If $\beta_3=0$, then $|\beta|\leq s-1$. Hence, we obtain by \eqref{L5} and Sobolev embedding inequality
\begin{equation*}
\begin{split}
\|(1+y)^{\ga+\al_3}D^{\beta}u\px D^{\al-\beta}\om\|_{L^2}
\leq &\|\py^{-1}D^{\beta}\om\|_{L^{\ty}}
\|(1+y)^{\ga+\al_3}\px D^{\al-\beta}\om\|_{L^2}\\
\leq &C\|D^{\beta}\om\|_{H_x^1L^{2}_{y,1}}
\|(1+y)^{\ga+\al_3}\px D^{\al-\beta}\om\|_{L^2}\\
\leq &C \|\om\|_{\mH^{|\beta|+1, 1}} \|\om\|_{\mH^{s, \ga}}\\
\leq &C \|\om\|_{\mH^{s, \ga}}^2.
\end{split}
\end{equation*}
Next,
\begin{equation}\label{chadianwangle}
\begin{split}
&\Big|\int_{\Om}(1+y)^{2\ga+2\al_3}D^{\al}\om D^{\beta}v\py D^{\al-\beta}\om dxdy\Big|\\
\leq
&\|(1+y)^{\ga+\al_3}D^{\beta}v\py D^{\al-\beta}\om\|_{L^2}
\|(1+y)^{\ga+\al_3}D^{\al}\om\|_{L^2}.
\end{split}
\end{equation}
If $\beta_3>0$, then by $\px u+\py v=0$, we have similarly
\begin{equation*}
\begin{split}
\|(1+y)^{\ga+\al_3}D^{\beta}v\py D^{\al-\beta}\om\|_{L^2}
=&\|(1+y)^{\ga+\al_3}D^{\beta-e_3+e_2}u D^{\al-\beta+e_3}\om\|_{L^2}\\
\leq&C \|\om\|_{\mH^{s, \ga}}^2.
\end{split}
\end{equation*}
If $\beta_3=0$, then $|\beta|\leq s-1$. Thus, for the case $|\beta|=s-1$, it gives by $\eqref{L3}_2$
\begin{equation}\label{dbvdeguji}
\begin{split}
\|(1+y)^{-1}D^{\beta}v\|_{L^2_x L^{\ty}_y}=
&\|(1+y)^{-1}\py^{-1}D^{\beta+e_2}u\|_{L^2_x L^{\ty}_y}\\
\leq&\|D^{\beta+e_2}u\|_{L^2_x L^{\ty}_y}\\
\leq&\|D^{\beta+e_2}(u-U)\|_{L^2_x L^{\ty}_y}+\|D^{\beta+e_2}U\|_{L^{2}_x}\\
\leq&\|D^{\beta+e_2}(u-U)\|_{L^2_x H^{1}_y}
+\|D^{\beta+e_2}U\|_{L^{2}_x}\\
\leq&\|u-U\|_{\mH^{s, 0}}+\|\om\|_{\mH^{s, 0}}+\|D^{\beta+e_2}U\|_{L^{2}_x}\\
\leq&\|u-U\|_{\mH^{s, \ga-1}}+\|\om\|_{\mH^{s, \ga}}+M,
\end{split}
\end{equation}
where we use the fact $\om=\py u$. Then
\begin{equation*}
\begin{split}
&\|(1+y)^{\ga+\al_3}D^{\beta}v\py D^{\al-\beta}\om\|_{L^2}
\\
\leq&\|(1+y)^{-1}D^{\beta}v\|_{L^2_x L^{\ty}_y}
\|(1+y)^{\ga+\al_3+1}\py D^{\al-\beta}\om\|_{L^{\ty}_x L^2_y}\\
\leq&C (\|u-U\|_{\mH^{s, \ga-1}}+\|\om\|_{\mH^{s, \ga}}+M)
\|(1+y)^{\ga+\al_3+1}\py D^{\al-\beta}\om\|_{H^{1}}\\
\leq&C (\|u-U\|_{\mH^{s, \ga-1}}+\|\om\|_{\mH^{s, \ga}}+M)
\|\om\|_{\mH^{s, \ga}}.
\end{split}
\end{equation*}
For the case $|\beta|\leq s-2$, similar to \eqref{dbvdeguji}, we can obtain
\begin{equation}\label{dbvdeguji2}
\begin{split}
\|(1+y)^{-1}D^{\beta}v\|_{L^{\ty}}
\leq\|u-U\|_{\mH^{s, \ga-1}}+\|\om\|_{\mH^{s, \ga}}+M.
\end{split}
\end{equation}
Hence, it follows
\begin{equation*}
\begin{split}
&\|(1+y)^{\ga+\al_3}D^{\beta}v\py D^{\al-\beta}\om\|_{L^2}
\\
\leq&\|(1+y)^{-1}D^{\beta}v\|_{L^{\ty}}
\|(1+y)^{\ga+\al_3+1}\py D^{\al-\beta}\om\|_{L^2}\\
\leq&C (\|u-U\|_{\mH^{s, \ga-1}}+\|\om\|_{\mH^{s, \ga}}+M)
\|\om\|_{\mH^{s, \ga}}.
\end{split}
\end{equation*}
Substituting the above inequalities into \eqref{K3jisuan} and \eqref{chadianwangle} yields \eqref{K3deguji}.

\hfill $\square$

$Proof~of~\eqref{K4deguji}$: When $s=4$, the result can be obtained by direct computation. Therefore, we only provide the proof for $s\geq5$ here, which requires the use of Lemma \ref{faa} and Proposition \ref{shuaijian}.

For $|\beta|=1$, using H\"{o}lder's inequality yields that
\begin{equation}\label{b1}
\begin{split}
&\Big|\int_{\Om}(1+y)^{2\ga+2\al_3}D^{\al}\om
D^{\beta}\om^{n-1}\py^2 D^{\al-\beta}\om dxdy\Big|\\
=&(1-n)\Big|\int_{\Om}(1+y)^{2\ga+2\al_3}D^{\al}\om
\cdot\om^{n-2}D^{\beta}\om\py^2 D^{\al-\beta}\om dxdy\Big|\\
\leq &C\|(1+y)^{-1+\beta_3}\om^{n-2}D^{\beta}\om\|_{L^{\ty}}
\|(1+y)^{\ga+\al_3}D^{\al}\om\|_{L^2}
\|(1+y)^{\ga+\al_3+1-\beta_3}\py D^{\al+e_3-\beta}\om\|_{L^2}.
\end{split}
\end{equation}
Combining \eqref{shuaijian5}, \eqref{K121} and \eqref{henyouyong} gives
\begin{equation*}
\|(1+y)^{-1+\beta_3}\om^{n-2}D^{\beta}\om\|_{L^{\ty}}\leq C.
\end{equation*}
Hence
\begin{equation}\label{b2}
\begin{split}
\Big|\int_{\Om}(1+y)^{2\ga+2\al_3}D^{\al}\om
D^{\beta}\om^{n-1}\py^2 D^{\al-\beta}\om dxdy\Big|
\leq C
\|\om\|_{\mH^{s, \ga}}
\|\py\om\|_{\mH^{s, \ga}}.
\end{split}
\end{equation}

If $1<|\beta|\leq s-2$, we can deduce
\begin{equation}\label{b3}
\begin{split}
&\Big|\int_{\Om}(1+y)^{2\ga+2\al_3}D^{\al}\om
D^{\beta}\om^{n-1}\py^2 D^{\al-\beta}\om dxdy\Big|\\
\leq &C\|(1+y)^{-2+\beta_3}D^{\beta}\om^{n-1}\|_{L^{\ty}}
\|(1+y)^{\ga+\al_3}D^{\al}\om\|_{L^2}
\|(1+y)^{\ga+\al_3+2-\beta_3}\py^2 D^{\al-\beta}\om\|_{L^2}.
\end{split}
\end{equation}
By using Lemma \ref{faa}, we have
\begin{equation}\label{K41}
D^{\beta}\om^{n-1} = \sum_{r=1}^{|\beta|} (n-1)^{\underline{r}} \, \om^{n-1-r} \sum_{\substack{\beta^1 + \cdots + \beta^r = \beta \\ |\beta^i| \ge 1}} \frac{\beta!}{r!} \prod_{i=1}^r \frac{D^{\beta^i} \om}{\beta^i!}.
\end{equation}
Let $r_0$ denote the number of $\beta^i$ such that $|\beta^i|=1$, then by \eqref{shuaijian5}, \eqref{K121} and $\beta^1 + \cdots + \beta^r = \beta$, we have
\begin{equation}\label{K43}
\begin{split}
\|(1+y)^{-2+\beta_3}\om^{n-1-r}\prod_{i=1}^r D^{\beta^i} \om\|_{L^{\ty}}
\leq C(1+\|\omega\|_{\mH^{s, \gamma}})^r
\|(1+y)^{c_r}\|_{L^{\ty}},
\end{split}
\end{equation}
where
\begin{equation}\label{K44}
\begin{split}
c_r=&-2-\si(n-1-r)-\si r_0-\frac{(s-2)\si(r-r_0)-2\ga(r-r_0)}{s-4}\\
&+\frac{(|\beta|-r_0)(\si-\ga)}{s-4}.
\end{split}
\end{equation}

Case 1: if $r=r_0$, then it must be that $r=r_0=|\beta|$. In this time, we have
\begin{equation}\label{K45}
\begin{split}
c_r=-2-\si(n-1)<0.
\end{split}
\end{equation}

Case 2: if $r-r_0\geq1$, then we have 
\begin{equation}\label{K46}
\begin{split}
c_r&=-2-\si(n-1)+\frac{(|\beta|-r)(\si-\ga)}{s-4}
-\frac{(\si-\ga)(r-r_0)}{s-4}\\
&\leq-2-\si(n-1)+\frac{(|\beta|-r)(\si-\ga)}{s-4}
-\frac{\si-\ga}{s-4}\\
&\leq-2-\si(n-1)+\frac{(s-3)(\si-\ga)}{s-4}
-\frac{\si-\ga}{s-4}\\
&=\si-\ga-2-\si(n-1)\\
&\leq\si-\ga-1\\
&\leq0.
\end{split}
\end{equation}
By combining \eqref{b3}-\eqref{K46}, we get that for $1<|\beta|\leq s-2$, there holds
\begin{equation}\label{xinhaolei}
\begin{split}
\Big|\int_{\Om}(1+y)^{2\ga+2\al_3}D^{\al}\om
D^{\beta}\om^{n-1}\py^2 D^{\al-\beta}\om dxdy\Big|
\leq C(1+\|\omega\|_{\mH^{s, \gamma}})^s.
\end{split}
\end{equation}

Now for $|\beta|=s-1$, when all $\beta^i$ in \eqref{K41} satisfy $|\beta^i|\leq s-2$, we can see that \eqref{xinhaolei} still holds. When there exists $|\beta^1|=s-1$, we have $r=1$ and $\beta^1=\beta$. Thus it holds
\begin{equation*}
\begin{split}
&\|(1+y)^{\ga+\al_3}\om^{n-1-r}\prod_{i=1}^r D^{\beta^i} \om
\py^2 D^{\al-\beta}\om
\|_{L^2}\\
=&\|(1+y)^{\ga+\al_3}\om^{n-2}D^{\beta} \om
\py^2 D^{\al-\beta}\om
\|_{L^2}\\
\leq&\|(1+y)^{-\ga-2-\si(n-2)}\|_{L^{\ty}}
\|(1+y)^{\si(n-2)}\om^{n-2}\|_{L^{\ty}}\\
&\cdot\|(1+y)^{\ga+\beta_3}D^{\beta}\om\|_{L^{\ty}_x L^{2}_y}
\|(1+y)^{\ga+\al_3-\beta_3+2}D^{\al-\beta+2e_3}\om\|_{L^2_x L^{\ty}_y}
\end{split}
\end{equation*}
By using $-\ga-2-\si(n-2)\leq\si-\ga-1\leq0$, \eqref{K121} and Sobolev embedding inequality, we obtain
\begin{equation*}
\begin{split}
&\|(1+y)^{\ga+\al_3}\om^{n-1-r}\prod_{i=1}^r D^{\beta^i} \om
\py^2 D^{\al-\beta}\om
\|_{L^2}\\
\leq&C\|(1+y)^{\ga+\beta_3}D^{\beta}\om\|_{H^1_x L^{2}_y}
\|(1+y)^{\ga+\al_3-\beta_3+2}D^{\al-\beta+2e_3}\om\|_{L^2_x H^{1}_y}\\
\leq&C\|\om\|_{\mH^{s, \gamma}}
\|\om\|_{\mH^{4, \gamma}}\\
\leq&C\|\om\|_{\mH^{s, \gamma}}^2.
\end{split}
\end{equation*}
Next for $|\beta|=s$, we can estimate by using the same method. In summary, we deduce that \eqref{K4deguji} holds by applying Cauchy's inequality.

\hfill $\square$

$Proof~of~\eqref{K5deguji}$: Similar to the proof of $\eqref{K4deguji}$, we also provide the proof here only for the case $s\geq5$.

First, we have
\begin{equation}\label{K51}
\begin{split}
D^{\al-\beta}|\py\om|^2=
\sum_{0\leq\tilde{\beta}\leq\al-\beta}
\binom{\al-\beta}{\tilde{\beta}}
D^{\tilde{\beta}}\py\om
D^{\al-\beta-\tilde{\beta}}\py\om.
\end{split}
\end{equation}
For $|\beta|\leq 1$, it easy to deduce by using $\om\in\mH^{s,\gamma}_{\si,\de}$
\begin{equation}\label{haolei}
|D^{\beta}\om^{n-2}|\leq C (1+y)^{-\si(n-2)-\beta_3}.
\end{equation}
By Sobolev embedding inequality, it gives that if $0<|\tilde{\beta}|\leq[\frac{s}{2}]$ 
\begin{equation*}
\begin{split}
&\|(1+y)^{2\ga+2+\al_3-\beta_3}D^{\tilde{\beta}}\py\om
D^{\al-\beta-\tilde{\beta}}\py\om\|_{L^2}\\
\leq 
&\|(1+y)^{\ga+1+\tilde{\beta}_3}D^{\tilde{\beta}}\py\om
\|_{L^{\ty}}
\|(1+y)^{\ga+1+\al_3-\beta_3-\tilde{\beta}_3}
D^{\al-\beta-\tilde{\beta}}\py\om\|_{L^2}\\
\leq &C \|\om\|_{\mH^{s, \gamma}}^2.
\end{split}
\end{equation*}
And if $|\tilde{\beta}|=0$,
\begin{equation*}
\begin{split}
&\|(1+y)^{\si+\ga+1+\al_3-\beta_3}\py\om
D^{\al-\beta}\py\om\|_{L^2}\\
\leq 
&\|(1+y)^{\si+1}\py\om
\|_{L^{\ty}}
\|(1+y)^{\ga+\al_3-\beta_3}
D^{\al-\beta}\py\om\|_{L^2}\\
\leq &C \|\py\om\|_{\mH^{s, \gamma}}.
\end{split}
\end{equation*}
If $[\frac{s}{2}]<|\tilde{\beta}|\leq|\al-\beta|$, it is similar. And since $\si+\ga+1<2\ga+2$, we have for $|\beta|\leq 1$
\begin{equation}\label{kuse}
\begin{split}
\|(1+y)^{\si+\ga+1+\al_3-\beta_3}D^{\al-\beta}|\py\om|^2\|_{L^2}
\leq C (\|\py\om\|_{\mH^{s, \gamma}}+\|\om\|_{\mH^{s, \gamma}}^2).
\end{split}
\end{equation}
Thus, there holds by \eqref{haolei}
\begin{equation}\label{miao}
\begin{split}
&\Big|\int_{\Om}(1+y)^{2\ga+2\al_3}D^{\al}\om
D^{\beta}\om^{n-2}D^{\al-\beta}|\py\om|^2 dxdy\Big|\\
\leq&
\|(1+y)^{-\si-1+\beta_3}D^{\beta}\om^{n-2}\|_{L^{\ty}} 
\|(1+y)^{\ga+\al_3}D^{\al}\om\|_{L^2}
\|(1+y)^{\si+\ga+1+\al_3-\beta_3}D^{\al-\beta}|\py\om|^2\|_{L^2}\\
\leq&
C\|(1+y)^{-\si-1-\si(n-2)}\|_{L^{\ty}} 
(\|\py\om\|_{\mH^{s, \gamma}}+\|\om\|_{\mH^{s, \gamma}}^2)\|\om\|_{\mH^{s, \gamma}}\\
\leq&
\ep\|\py\omega\|_{\mH^{s, \gamma}}^2
+C\ep^{-1}(1+\|\omega\|_{\mH^{s, \gamma}})^{4},
\end{split}
\end{equation}
here, for the last inequality, we have used \eqref{henyouyong}.

For $1<|\beta|\leq s-2$, we know that at least one of $|\tilde{\beta}+e_3|$ and $|\al-\beta-\tilde{\beta}+e_3|$ does not exceed $[\frac{s}{2}]$, and we may assume $|\tilde{\beta}+e_3|\leq[\frac{s}{2}]$. In this time, we can obtain by \eqref{shuaijian5},
\begin{equation}\label{haoleia}
|D^{\tilde{\beta}}\py\om|\leq C(1+\|\om\|_{\mH^{s, \gamma}}) (1+y)^{-\frac{\si+\ga}{2}-1-\tilde{\beta}_3}.
\end{equation}
Therefore, 
\begin{equation*}
\begin{split}
&\|(1+y)^{\frac{\si}{2}+\frac{3}{2}\ga+2+\al_3-\beta_3}D^{\tilde{\beta}}\py\om
D^{\al-\beta-\tilde{\beta}}\py\om\|_{L^2}\\
\leq 
&\|(1+y)^{\frac{\si+\ga}{2}+1+\tilde{\beta}_3}D^{\tilde{\beta}}\py\om
\|_{L^{\ty}}
\|(1+y)^{\ga+1+\al_3-\beta_3-\tilde{\beta}_3}
D^{\al-\beta-\tilde{\beta}}\py\om\|_{L^2}\\
\leq &C \|\om\|_{\mH^{s, \gamma}}(1+\|\om\|_{\mH^{s, \gamma}}),
\end{split}
\end{equation*}
Then
\begin{equation*}
\begin{split}
\|(1+y)^{\frac{\si}{2}+\frac{3}{2}\ga+2+\al_3-\beta_3}
D^{\al-\beta}|\py\om|^2\|_{L^2}
\leq C \|\om\|_{\mH^{s, \gamma}}(1+\|\om\|_{\mH^{s, \gamma}}),
\end{split}
\end{equation*}
On the other hand, similar to \eqref{K41}-\eqref{K46}, we obtain
\begin{equation}\label{K410}
D^{\beta}\om^{n-2} = \sum_{r=1}^{|\beta|} (n-2)^{\underline{r}} \, \om^{n-2-r} \sum_{\substack{\beta^1 + \cdots + \beta^r = \beta \\ |\beta^i| \ge 1}} \frac{\beta!}{r!} \prod_{i=1}^r \frac{D^{\beta^i} \om}{\beta^i!}.
\end{equation}
And
\begin{equation}\label{K430}
\begin{split}
\|(1+y)^{-\frac{\si+\ga}{2}-2+\beta_3}\om^{n-2-r}\prod_{i=1}^r D^{\beta^i} \om\|_{L^{\ty}}
\leq C(1+\|\omega\|_{\mH^{s, \gamma}})^r
\|(1+y)^{d_r}\|_{L^{\ty}},
\end{split}
\end{equation}
where
\begin{equation}\label{K440}
\begin{split}
d_r=&-\frac{\si+\ga}{2}-2-\si(n-2-r)-\si r_0-\frac{(s-2)\si(r-r_0)-2\ga(r-r_0)}{s-4}\\
&+\frac{(|\beta|-r_0)(\si-\ga)}{s-4},
\end{split}
\end{equation}
and $r_0$ denote the number of $\beta^i$ such that $|\beta^i|=1$.

Case 1: if $r=r_0$, then it must be that $r=r_0=|\beta|$. In this time, we have
\begin{equation}\label{K450}
\begin{split}
d_r=-\frac{\si+\ga}{2}-2-\si(n-2)<\frac{1}{2}(\si-\ga-1)<0.
\end{split}
\end{equation}

Case 2: if $r-r_0\geq1$, then we have 
\begin{equation}\label{K460}
\begin{split}
c_r&=-\frac{\si+\ga}{2}-2-\si(n-2)
+\frac{(|\beta|-r)(\si-\ga)}{s-4}
-\frac{(\si-\ga)(r-r_0)}{s-4}\\
&\leq-\frac{\si+\ga}{2}-2-\si(n-2)
+\frac{(|\beta|-r)(\si-\ga)}{s-4}
-\frac{\si-\ga}{s-4}\\
&\leq-\frac{\si+\ga}{2}-2-\si(n-2)
+\frac{(s-3)(\si-\ga)}{s-4}
-\frac{\si-\ga}{s-4}\\
&=\frac{3}{2}(\si-\ga-\frac{2}{3})-1-\si(n-1)\\
&\leq\frac{3}{2}(\si-\ga-\frac{2}{3})\\
&\leq0.
\end{split}
\end{equation}

Then, by combining \eqref{K410}-\eqref{K460}, we have for $1<|\beta|\leq s-2$, there holds
\begin{equation}\label{xinhaolei2}
\begin{split}
&\Big|\int_{\Om}(1+y)^{2\ga+2\al_3}D^{\al}\om
D^{\beta}\om^{n-2}D^{\al-\beta}|\py\om|^2 dxdy\Big|\\
\leq &\|(1+y)^{-\frac{\si+\ga}{2}-2+\beta_3}D^{\beta}\om^{n-2}\|_{L^{\ty}}
\|(1+y)^{\ga+\al_3}D^{\al}\om\|_{L^2}
\|(1+y)^{\frac{\si}{2}+\frac{3}{2}\ga+2+\al_3-\beta_3}
D^{\al-\beta}|\py\om|^2\|_{L^2}\\
\leq &C(1+\|\omega\|_{\mH^{s, \gamma}})^{s+1}.
\end{split}
\end{equation}
Finally, for $|\beta|=s-1$ or $s$, one may refer to the estimate of $K_4$; for simplicity, we omit it. In summary, we deduce that \eqref{K5deguji} holds by applying Cauchy's inequality.

\hfill $\square$

$Proof~of~\eqref{K140}$: We consider the following two cases.

Case 1. $|\al|\leq s-1$. By using inequality \eqref{L1}, we deduce that
\begin{equation}\label{K141}
\begin{split}
|K^4_1|
&=n\Big|\int_{\T}(\om^{n-1} D^{\al}\om
\py D^{\al}\om)|_{y=0}dx\Big|\\
&\leq
n\|\om^{n-1}D^{\al}\om\|_{L^2}\|\py^2 D^{\al}\om\|_{L^2}
+n\|\py(\om^{n-1}D^{\al}\om)\|_{L^2}\|\py D^{\al}\om\|_{L^2}.
\end{split}
\end{equation}
Using \eqref{K121} and \eqref{henyouyong} yields
\begin{equation}\label{K142}
|\om^{n-1}|\leq\de^{n-1}(1+y)^{-\si(n-1)} \leq \de^{n-1}(1+y).
\end{equation}
On the other hand, by \eqref{K124} and \eqref{K142}, we have
\begin{equation}\label{K143}
\begin{split}
|\py(\om^{n-1}D^{\al}\om)|
&=|(n-1)\om^{n-2}\py\om D^{\al}\om+\om^{n-1}\py D^{\al}\om|\\
&\leq(1-n)\de^{n-3}|D^{\al}\om|+\de^{n-1}(1+y)|\py D^{\al}\om|.
\end{split}
\end{equation}
Substituting \eqref{K142} and \eqref{K143} into \eqref{K143} yields
\begin{equation}\label{K144}
\begin{split}
|K^4_1|\leq
&n\de^{n-1}\|(1+y)D^{\al}\om\|_{L^2}\|\py^2 D^{\al}\om\|_{L^2}\\
&+n(1-n)\de^{n-3}\|D^{\al}\om\|_{L^2}\|\py D^{\al}\om\|_{L^2}\\
&+n\de^{n-1}\|(1+y)\py D^{\al}\om\|_{L^2}\|\py D^{\al}\om\|_{L^2}.
\end{split}
\end{equation}
Using the fact that $\ga\geq1$ yields 
\begin{equation}\label{K145}
\begin{split}
|K^4_1|\leq
&n\de^{n-1}\|(1+y)^{\ga}D^{\al}\om\|_{L^2}\|\py^2 D^{\al}\om\|_{L^2}\\
&+n(1-n)\de^{n-3}\|D^{\al}\om\|_{L^2}\|\py D^{\al}\om\|_{L^2}\\
&+n\de^{n-1}\|(1+y)^{\ga}\py D^{\al}\om\|_{L^2}\|\py D^{\al}\om\|_{L^2}\\
\leq&\ep\|\py\omega\|_{\mH^{s, \gamma}}^2
+C\ep^{-1}\|\omega\|_{\mH^{s, \gamma}}^2,
\end{split}
\end{equation}
where we use the fact that $|\al|\leq s-1$.

Case 2. $|\al|=\al_1+\al_2+\al_3=s$. Since $\al_1+\al_2\leq s-1$, we have
$\al_3\geq1$. Then we can set $\tilde{\al}:=\al-e_3=(\al_1,\al_2,\al_3-1)$ with $|\tilde{\al}|=s-1$. At this point, we have $\py D^{\al}=\py^2 D^{\tilde{\al}}$. Hence the equation \eqref{alwodu} reads
\begin{equation}\label{K146}
\begin{split}
&n\om^{n-1}\py D^{\al}\om\\
=&\pt D^{\tilde{\al}}\om +u\px D^{\tilde{\al}}\om +v\py D^{\tilde{\al}}\om+\sum_{0<\beta\leq\tilde{\al}}\binom{\tilde{\al}}{\beta}
\left(D^{\beta}u\px D^{\tilde{\al}-\beta}\om+D^{\beta}v\py D^{\tilde{\al}-\beta}\om\right)\\
&-n\sum_{0<\beta\leq\tilde{\al}}\binom{\tilde{\al}}{\beta}
D^{\beta}\om^{n-1}\py^2 D^{\tilde{\al}-\beta}\om
-n(n-1)\sum_{0\leq\beta\leq\tilde{\al}}\binom{\tilde{\al}}{\beta}
D^{\beta}\om^{n-2}D^{\tilde{\al}-\beta}|\py\om|^2.
\end{split}
\end{equation}
By using \eqref{L1}, it is easy to get 
\begin{equation}\label{K147}
\begin{split}
&\Big|\int_{\T}\pt D^{\tilde{\al}}\om D^{\al}\om|_{y=0}dx\Big|\\
\leq&\|\py\pt D^{\tilde{\al}}\om\|_{L^2}\|D^{\al}\om\|_{L^2}
+\|\pt D^{\tilde{\al}}\om\|_{L^2}\|\py D^{\al}\om\|_{L^2}\\
\leq&\ep\|\py\omega\|_{\mH^{s, \gamma}}^2
+C\ep^{-1}\|\omega\|_{\mH^{s, \gamma}}^2.
\end{split}
\end{equation}
Since the boundary condition $\eqref{feiniu2}_4$, we have at $y=0$
\begin{equation*}
u\px D^{\tilde{\al}}\om +v\py D^{\tilde{\al}}\om=0.
\end{equation*}
Next, using \eqref{L1} yields
\begin{equation}\label{K149}
\begin{split}
&\Big|\int_{\T}D^{\beta}u\px D^{\tilde{\al}-\beta}\om D^{\al}\om|_{y=0}dx\Big|\\
\leq&\|\py D^{\beta}u\px D^{\tilde{\al}-\beta}\om
+D^{\beta}u\py\px D^{\tilde{\al}-\beta}\om\|_{L^2}\|D^{\al}\om\|_{L^2}
+\|D^{\beta}u\px D^{\tilde{\al}-\beta}\om\|_{L^2}\|\py D^{\al}\om\|_{L^2}.
\end{split}
\end{equation}
Noticing that $\om=\py u$ (or $u=\py^{-1}\om$), we have by \eqref{L2}
\begin{equation}\label{K1410}
\begin{split}
\|\py D^{\beta}u\px D^{\tilde{\al}-\beta}\om
\|_{L^2}
=\|D^{\beta}\om D^{\tilde{\al}-\beta+e_2}\om
\|_{L^2}
\leq C \|\omega\|_{\mH^{s, 0}}^2.
\end{split}
\end{equation}
For $D^{\beta}u\py\px D^{\tilde{\al}-\beta}\om$, if $\beta_3\neq0$, we similarly have
\begin{equation}\label{K1411}
\begin{split}
\|D^{\beta}u\py\px D^{\tilde{\al}-\beta}\om\|_{L^2}
=\|D^{\beta-e_3}\om D^{\tilde{\al}-\beta+e_2}\py\om\|_{L^2}
\leq C \|\omega\|_{\mH^{s-1, 0}}\|\py\omega\|_{\mH^{s-1, 0}}.
\end{split}
\end{equation}
On the other hand, if $\beta_3=0$, then by \eqref{L4.5}, it follows 
\begin{equation}\label{K1412}
\begin{split}
\|D^{\beta}u\py\px D^{\tilde{\al}-\beta}\om\|_{L^2}
=\|D^{\beta}\py^{-1}\om D^{\tilde{\al}-\beta+e_2}\py\om\|_{L^2}
\leq C \|\omega\|_{\mH^{s, 0}}\|\py\omega\|_{\mH^{s, 1}}.
\end{split}
\end{equation}
Combining \eqref{K1411} and \eqref{K1412}, we get
\begin{equation}\label{K1413}
\begin{split}
\|D^{\beta}u\py\px D^{\tilde{\al}-\beta}\om\|_{L^2}
\leq C \|\omega\|_{\mH^{s, \ga}}\|\py\omega\|_{\mH^{s, \ga}}.
\end{split}
\end{equation}
Likely \eqref{K1413}, we can obtain
\begin{equation}\label{K1414}
\begin{split}
\|D^{\beta}u\px D^{\tilde{\al}-\beta}\om\|_{L^2}
\leq C \|\omega\|_{\mH^{s, \ga}}^2.
\end{split}
\end{equation}
Substituting the above inequalities into \eqref{K149} and then applying Cauchy's inequality, we can obtain
\begin{equation}\label{K1415}
\begin{split}
\Big|\int_{\T}D^{\beta}u\px D^{\tilde{\al}-\beta}\om D^{\al}\om|_{y=0}dx\Big|\leq\ep\|\py\omega\|_{\mH^{s, \gamma}}^2
+C\ep^{-1}\|\omega\|_{\mH^{s, \gamma}}^2(1+\|\omega\|_{\mH^{s, \gamma}}^2).
\end{split}
\end{equation}
Now, by using \eqref{L1}, we have
\begin{equation}\label{K1416}
\begin{split}
&\Big|\int_{\T}D^{\beta}v\py D^{\tilde{\al}-\beta}\om D^{\al}\om|_{y=0}dx\Big|\\
\leq&\|\py D^{\beta}v\py D^{\tilde{\al}-\beta}\om
+D^{\beta}v\py^2 D^{\tilde{\al}-\beta}\om\|_{L^2}\|D^{\al}\om\|_{L^2}
+\|D^{\beta}v\py D^{\tilde{\al}-\beta}\om\|_{L^2}\|\py D^{\al}\om\|_{L^2}.
\end{split}
\end{equation}
Next, using a similar estimation method as for $K_3$, we obtain 
\begin{equation}\label{K1423}
\begin{split}
\Big|\int_{\T}D^{\beta}v\py D^{\tilde{\al}-\beta}\om D^{\al}\om|_{y=0}dx\Big|\leq\ep\|\py\omega\|_{\mH^{s, \gamma}}^2
+C\ep^{-1}(\|u-U\|_{\mH^{s, \ga-1}}^2 +\|\om\|_{\mH^{s, \ga}}^2+
M^2)^2.
\end{split}
\end{equation}
For simplicity, we omit the details. And similarly, analogous to $K_4$ and $K_5$, we obtain
\begin{equation}\label{K1435}
\begin{split}
\Big|\int_{\T}D^{\beta}\om^{n-1}\py^2 D^{\tilde{\al}-\beta}\om D^{\al}\om|_{y=0}dx\Big|\leq\ep\|\py\omega\|_{\mH^{s, \gamma}}^2
+C\ep^{-1}(1+\|\omega\|_{\mH^{s, \gamma}})^{2s},
\end{split}
\end{equation}
\begin{equation}\label{K1442}
\begin{split}
\Big|\int_{\T}D^{\beta}\om^{n-2}D^{\tilde{\al}-\beta}|\py\om|^2 D^{\al}\om|_{y=0}dx\Big|\leq\ep\|\py\omega\|_{\mH^{s, \gamma}}^2
+C\ep^{-1}(1+\|\omega\|_{\mH^{s, \gamma}})^{2s+2}.
\end{split}
\end{equation}
Thus far, we have proved \eqref{K140}.

\hfill $\square$

\subsection{Estimates Only in Tangential Derivatives}\label{zuoqiexiangguji}
In this subsection, we will deal with the estimates for $D^{\al}\om$ with $\al_1+\al_2=s$.

Since $\om=\py u$, the system $\eqref{feiniu2}_1$ can be written as the following
\begin{equation}\label{feiniu3}
\pt u +u\px u +v\py u-n\om^{n-1}\py^2 u+\px p=0.
\end{equation}
Then by using Bernoulli's law \eqref{Bernoulli}, we have
\begin{equation}\label{feiniu4}
\pt (u-U) +u\px (u-U) +v\py (u-U)-n\om^{n-1}\py^2 (u-U)+(u-U)\px U=0.
\end{equation}

From now on, for ease of expression, if $\al_{3}=0$, we will use $\ptau^{\al}$ to denote $D^{\al}$.

Applying the operator $\p_{\tau}^{\al}$ with $|\al|=s$ to \eqref{feiniu4}, it yields
\begin{equation}\label{ptauu}
\begin{split}
&\pt \p_{\tau}^{\al}(u-U) +u\px \p_{\tau}^{\al}(u-U) +v\py \p_{\tau}^{\al}(u-U)-n\om^{n-1}\py^2 \p_{\tau}^{\al}(u-U)
+\p_{\tau}^{\al}v\om\\
=
&-\sum_{0<\beta\leq\al}\binom{\al}{\beta}
\p_{\tau}^{\beta}u\px \p_{\tau}^{\al-\beta}(u-U)
-\sum_{0<\beta<\al}\binom{\al}{\beta}
\p_{\tau}^{\beta}v\py \p_{\tau}^{\al-\beta}(u-U)\\
&+n\sum_{0<\beta\leq\al}\binom{\al}{\beta}
\p_{\tau}^{\beta}\om^{n-1}\py^2 \p_{\tau}^{\al-\beta}(u-U)
-\sum_{0\leq\beta\leq\al}\binom{\al}{\beta}
\p_{\tau}^{\beta}(u-U)\px\p_{\tau}^{\al-\beta}U.
\end{split}
\end{equation}
On the other hand, we have
\begin{equation}\label{ptauwodu}
\begin{split}
&\pt \p_{\tau}^{\al}\om +u\px \p_{\tau}^{\al}\om +v\py \p_{\tau}^{\al}\om-n\om^{n-1}\py^2 \p_{\tau}^{\al}\om
+\p_{\tau}^{\al}v\py\om\\
=
&-\sum_{0<\beta\leq\al}\binom{\al}{\beta}
\p_{\tau}^{\beta}u\px \p_{\tau}^{\al-\beta}\om
-\sum_{0<\beta<\al}\binom{\al}{\beta}
\p_{\tau}^{\beta}v\py \p_{\tau}^{\al-\beta}\om\\
&+n\sum_{0<\beta\leq\al}\binom{\al}{\beta}
\p_{\tau}^{\beta}\om^{n-1}\py^2 \p_{\tau}^{\al-\beta}\om
+n(n-1)\sum_{0\leq\beta\leq\al}\binom{\al}{\beta}
\p_{\tau}^{\beta}\om^{n-2}\p_{\tau}^{\al-\beta}|\py\om|^2.
\end{split}
\end{equation}
Subtracting $\frac{\py\om}{\om}\times\eqref{ptauu}$ from \eqref{ptauwodu} yields
\begin{equation}\label{ptaug}
\begin{split}
&\pt g_{\al} +u\px g_{\al} +v\py g_{\al}-n\om^{n-1}\py^2 g_{\al}
\\
=&-\sum_{0<\beta<\al}\binom{\al}{\beta}
\p_{\tau}^{\beta}u g_{\al-\beta+e_2}
-\sum_{0<\beta<\al}\binom{\al}{\beta}
\p_{\tau}^{\beta}v(\py \p_{\tau}^{\al-\beta}\om-a\p_{\tau}^{\al-\beta}\om)\\
&+n\sum_{0<\beta\leq\al}\binom{\al}{\beta}
\p_{\tau}^{\beta}\om^{n-1}(\py^2 \p_{\tau}^{\al-\beta}\om
-a\py \p_{\tau}^{\al-\beta}\om)\\
&+n(n-1)\sum_{0\leq\beta\leq\al}\binom{\al}{\beta}
\p_{\tau}^{\beta}\om^{n-2}\p_{\tau}^{\al-\beta}|\py\om|^2
+\sum_{0\leq\beta<\al}\binom{\al}{\beta}
a\p_{\tau}^{\beta}(u-U)\px\p_{\tau}^{\al-\beta}U\\
&-\ptau^{\al}(u-U)\{2n(n-1)a\om^{n-2}\py^2\om+na\py^2\om^{n-1}
-n(n-1)a^3\om^{n-1}\}\\
&+2n\om^{n-1}\py a g_{\al}-g_{e_2}\ptau U.
\end{split}
\end{equation}
where we set $g_{\al}:=\ptau^{\al}\om-a\ptau^{\al}(u-U)$ for any $\al=(\al_1,\al_2,0)$ and $a:=\frac{\py\om}{\om}$.
The derivation of \eqref{ptaug} will be provided later. Now, we are going to prove the following estimate for $g_{\al}$.
\begin{Proposition}[Weighted estimates for $g_{\al}$ with $\al_1+\al_2=s$ and $\al_3=0$]\label{guji2}
Assume $\frac{1}{3}<n<1$. Let $s\geq4,~\ga\geq1,~\ga+\frac{1}{2}<\si
\leq\min\{\frac{1}{1-n},\ga+\frac{2}{3}\}$ and
$\de\in(0,1)$ is small enough. If $(u,v,\om)$ is a classical solution of \eqref{wodu} in $[0,T]$ and satisfies
\begin{equation}\label{jiashe22}
\om\in L^{\ty}(0,T;\mH^{s,\gamma}_{\si,\de}),~\py\om\in L^{2}(0,T;\mH^{s,\gamma}),
\end{equation}
then there exists a positive constant C, which depends on $n,~s,~\ga,~\si$ and $\de$ such that for any small $0<\ep<1$,
\begin{equation}\label{faguji}
\begin{split}
&\sum_{\al_1+\al_2=s}
\left(\frac{d}{dt}\|g_{\al}\|^2_{L^2_{\ga}}
+ n\de^{1-n}\|\py g_{\al}\|^2_{L^2_{\ga}}
\right)\\
\leq &C\ep(\|\py\omega\|_{\mH^{s, \gamma}}^2
+\sum_{\al_1+\al_2=s}\|\py g_{\al}\|^2_{L^2_{\ga}})+C\ep^{-1}(1+M+\|u-U\|_{\mH^{s, \ga-1}}\\
&\quad\quad\quad\quad\quad\quad\quad\quad\quad
\quad\quad\quad\quad\quad\quad\quad\quad\quad
+\|\omega\|_{\mH^{s, \gamma}}+\sum_{\al_1+\al_2=s}\|g_{\al}\|^2_{L^2_{\ga}})^{2s+2}.
\end{split}
\end{equation}
\end{Proposition}

\textbf{Proof of Proposition \ref{guji2}.}
Taking $L^2$ inner product of \eqref{ptaug} with $(1+y)^{2\ga}g_{\al}$ where $\al_1+\al_2=s$ and $\al_3=0$ yields
\begin{equation}\label{haoweixian}
\begin{split}
&\frac{1}{2}\frac{d}{dt}\|(1+y)^{\ga}g_{\al}\|_{L^2}^2\\
=&n\int_{\Om}(1+y)^{2\ga}\om^{n-1}g_{\al}\py^2 g_{\al} dxdy-\int_{\Om}(1+y)^{2\ga}g_{\al}(u\px g_{\al} +v\py g_{\al}) dxdy
\\
&-\sum_{0<\beta<\al}\binom{\al}{\beta}\int_{\Om}(1+y)^{2\ga}g_{\al}
\p_{\tau}^{\beta}u g_{\al-\beta+e_2} dxdy\\
&-\sum_{0<\beta<\al}\binom{\al}{\beta}\int_{\Om}(1+y)^{2\ga}g_{\al}
\p_{\tau}^{\beta}v(\py \p_{\tau}^{\al-\beta}\om-a\p_{\tau}^{\al-\beta}\om) dxdy\\
&+n\sum_{0<\beta\leq\al}\binom{\al}{\beta}
\int_{\Om}(1+y)^{2\ga}g_{\al}
\p_{\tau}^{\beta}\om^{n-1}(\py^2 \p_{\tau}^{\al-\beta}\om
-a\py \p_{\tau}^{\al-\beta}\om) dxdy\\
&+n(n-1)\sum_{0\leq\beta\leq\al}\binom{\al}{\beta}
\int_{\Om}(1+y)^{2\ga}g_{\al}
\p_{\tau}^{\beta}\om^{n-2}\p_{\tau}^{\al-\beta}|\py\om|^2 dxdy\\
&+\sum_{0\leq\beta<\al}\binom{\al}{\beta}
\int_{\Om}(1+y)^{2\ga}g_{\al}
a\p_{\tau}^{\beta}(u-U)\px\p_{\tau}^{\al-\beta}U dxdy\\
&-\int_{\Om}(1+y)^{2\ga}g_{\al}
\ptau^{\al}(u-U)\{2n(n-1)a\om^{n-2}\py^2\om+na\py^2\om^{n-1}
-n(n-1)a^3\om^{n-1}\} dxdy\\
&+2n\int_{\Om}(1+y)^{2\ga}\om^{n-1}\py a |g_{\al}|^2 dxdy- \int_{\Om}(1+y)^{2\ga}g_{\al}g_{e_2}\ptau^{\al} U dxdy\\
:=&\sum\limits_{i=1}^{10} S_i.
\end{split}
\end{equation}

For $S_1$, by integration by parts, we have
\begin{equation}\label{S1}
\begin{split}
S_1=&-n\int_{\Om}(1+y)^{2\ga}\om^{n-1}|\py g_{\al}|^2 dxdy\\
&-n(n-1)\int_{\Om}(1+y)^{2\ga}\py\om\om^{n-2}g_{\al}\py g_{\al} dxdy\\
&-2n\ga\int_{\Om}(1+y)^{2\ga-1}\om^{n-1}g_{\al}\py g_{\al} dxdy\\
&-n\int_{\T}(\om^{n-1} g_{\al}\py g_{\al} )|_{y=0}dx\\
:=&\sum\limits_{i=1}^{4} S_1^i.
\end{split}
\end{equation}
For $S_1^i,~i=1,2$ and $3$, similar to $K_1^i$ in Proposition \ref{guji1}, we can deduce
\begin{equation}\label{S11}
\begin{split}
&S_1^1+S_1^2+S_1^3\\
\leq &-n\de^{1-n}\|(1+y)^{\ga}\py g_{\al}\|^2_{L^2}+C\ep\|(1+y)^{\ga}\py g_{\al}\|^2_{L^2}
+C\ep^{-1}\|(1+y)^{\ga} g_{\al}\|^2_{L^2}.
\end{split}
\end{equation}
For $S_1^4$, we have 
\begin{equation}\label{S14}
\begin{split}
|S_1^4|\leq \ep\|\py\omega\|_{\mH^{s, \gamma}}^2
+C\ep^{-1}(1+M+\|\omega\|_{\mH^{s, \gamma}})^{2s}
\end{split}
\end{equation}
The proof of \eqref{S14} will be provided later. Hence, we have the following estimate
\begin{equation}\label{S1dekongzhi}
\begin{split}
S_1 \leq &-n\de^{1-n}\|(1+y)^{\ga}\py g_{\al}\|^2_{L^2}+C\ep(\|(1+y)^{\ga}\py g_{\al}\|^2_{L^2}+
\|\py\omega\|_{\mH^{s, \gamma}}^2)\\
&+C\ep^{-1}\|(1+y)^{\ga} g_{\al}\|^2_{L^2}
+C\ep^{-1}(1+M+\|\omega\|_{\mH^{s, \gamma}})^{2s}.
\end{split}
\end{equation}

For $S_2$, by integration by parts, we have
\begin{equation}\label{S2dekongzhi}
\begin{split}
|S_2|=&\ga|\int_{\Om}(1+y)^{2\ga-1}v |g_{\al}|^2 dxdy|\\
\leq&\ga \|(1+y)^{-1}v\|_{L^{\ty}}
\|(1+y)^{\ga}g_{\al}\|_{L^2}^2\\
\leq& C(\|u-U\|_{\mH^{s, \gamma-1}}+M)\|(1+y)^{\ga}g_{\al}\|_{L^2}^2,
\end{split}
\end{equation}
where we use \eqref{vdeguji}.

For $S_3$, 
we obtain by using Sobolev embedding inequality, and \eqref{L5}
\begin{equation}\label{aiyouwei}
\begin{split}
&\Big|\int_{\Om}(1+y)^{2\ga}g_{\al}
\p_{\tau}^{\beta}u g_{\al-\beta+e_2} dxdy\Big|\\
\leq&\|(1+y)^{\ga}g_{\al}\|_{L^2}
\|(1+y)^{\ga}\p_{\tau}^{\beta}u g_{\al-\beta+e_2}\|_{L^2}
\\
\leq&\|(1+y)^{\ga}g_{\al}\|_{L^2}
\|\py^{-1}\p_{\tau}^{\beta}\om\|_{L^{\ty}}
\|(1+y)^{\ga} g_{\al-\beta+e_2}\|_{L^2}\\
\leq&C \|(1+y)^{\ga}g_{\al}\|_{L^2}
\|\om\|_{\mH^{|\beta|+1, 1}}
\|(1+y)^{\ga} g_{\al-\beta+e_2}\|_{L^2}
\\
\leq&C \|(1+y)^{\ga}g_{\al}\|_{L^2}
\|\om\|_{\mH^{s, \ga}}
\|(1+y)^{\ga} g_{\al-\beta+e_2}\|_{L^2},
\end{split}
\end{equation}
where we use $u=\py^{-1}\om$. On the other hand, we have 
\begin{equation}\label{aaa}
\begin{split}
|a|=|\frac{\py\om}{\om}|
\leq\de^{-2}(1+y)^{-1}.
\end{split}
\end{equation}
Since $g_{\al-\beta+e_2}:=
\ptau^{\al-\beta+e_2}\om-a\ptau^{\al-\beta+e_2}(u-U)$, we deduce
\begin{equation}\label{ggg}
\begin{split}
&\|(1+y)^{\ga} g_{\al-\beta+e_2}\|_{L^2}\\
=&\|(1+y)^{\ga}(\ptau^{\al-\beta+e_2}\om-
a\ptau^{\al-\beta+e_2}(u-U))\|_{L^2}\\
\leq& \|(1+y)^{\ga}\ptau^{\al-\beta+e_2}\om\|_{L^2}
+\|(1+y)a\|_{L^{\ty}}\|(1+y)^{\ga-1}
\ptau^{\al-\beta+e_2}(u-U)\|_{L^2}\\
\leq& C (\|\om\|_{\mH^{s, \ga}}+\|u-U\|_{\mH^{s, \ga-1}}).
\end{split}
\end{equation}
Combining \eqref{aiyouwei} and \eqref{ggg} gives
\begin{equation}\label{S3dekongzhi}
\begin{split}
|S_3|\leq C \|(1+y)^{\ga}g_{\al}\|_{L^2}
\|\om\|_{\mH^{s, \ga}}(\|\om\|_{\mH^{s, \ga}}+\|u-U\|_{\mH^{s, \ga-1}}).
\end{split}
\end{equation}

For $S_4$, 
by using H\"{o}lder's inequality and \eqref{dbvdeguji}, we obtain for $|\beta|=s-1$
\begin{equation*}
\begin{split}
&\Big|
\int_{\Om}(1+y)^{2\ga}g_{\al}
\p_{\tau}^{\beta}v(\py \p_{\tau}^{\al-\beta}\om-a\p_{\tau}^{\al-\beta}\om) dxdy\Big|
\\
\leq &C \|(1+y)^{\ga}g_{\al}\|_{L^2}
\|(1+y)^{-1}\p_{\tau}^{\beta}v\|_{L^2_x L^{\ty}_y}\\
&~~~\cdot(
\|(1+y)^{\ga+1}\py \p_{\tau}^{\al-\beta}\om\|_{L^{\ty}_x L^2_y}
+\|(1+y)a\|_{L^{\ty}}
\|(1+y)^{\ga}\p_{\tau}^{\al-\beta}\om\|_{L^{\ty}_x L^2_y})\\
\leq &C \|(1+y)^{\ga}g_{\al}\|_{L^2}(\|u-U\|_{\mH^{s, \ga-1}}+\|\om\|_{\mH^{s, \ga}}+M)\\
&~~~\cdot(
\|(1+y)^{\ga+1}\py \p_{\tau}^{\al-\beta}\om\|_{H^1}
+\|(1+y)a\|_{L^{\ty}}
\|(1+y)^{\ga}\p_{\tau}^{\al-\beta}\om\|_{H^1})\\
\leq &C \|(1+y)^{\ga}g_{\al}\|_{L^2}(\|u-U\|_{\mH^{s, \ga-1}}+\|\om\|_{\mH^{s, \ga}}+M)\|\om\|_{\mH^{s, \ga}}.
\end{split}
\end{equation*}
For the case $0<|\beta|\leq s-2$, using \eqref{dbvdeguji2} gives
\begin{equation*}
\begin{split}
&\Big|
\int_{\Om}(1+y)^{2\ga}g_{\al}
\p_{\tau}^{\beta}v(\py \p_{\tau}^{\al-\beta}\om-a\p_{\tau}^{\al-\beta}\om) dxdy\Big|
\\
\leq &C \|(1+y)^{\ga}g_{\al}\|_{L^2}
\|(1+y)^{-1}\p_{\tau}^{\beta}v\|_{L^{\ty}}\\
&~~~\cdot(
\|(1+y)^{\ga+1}\py \p_{\tau}^{\al-\beta}\om\|_{L^2}
+\|(1+y)a\|_{L^{\ty}}
\|(1+y)^{\ga}\p_{\tau}^{\al-\beta}\om\|_{L^2})\\
\leq &C \|(1+y)^{\ga}g_{\al}\|_{L^2}(\|u-U\|_{\mH^{s, \ga-1}}+\|\om\|_{\mH^{s, \ga}}+M)\|\om\|_{\mH^{s, \ga}}.
\end{split}
\end{equation*}
Thus, we have an estimate for $S_4$
\begin{equation}\label{S4dekongzhi}
\begin{split}
|S_4|
\leq C \|(1+y)^{\ga}g_{\al}\|_{L^2}(\|u-U\|_{\mH^{s, \ga-1}}+\|\om\|_{\mH^{s, \ga}}+M)\|\om\|_{\mH^{s, \ga}}.
\end{split}
\end{equation}

For $S_5$ and $S_6$, similar to $K_4$ and $K_5$ in Proposition \ref{guji1}, we can deduce by using \eqref{aaa}
\begin{equation}\label{S5deguji}
\begin{split}
|S_5|\leq \ep\|\py\omega\|_{\mH^{s, \gamma}}^2
+C\ep^{-1}(1+\|\omega\|_{\mH^{s, \gamma}}+\|(1+y)^{\ga}g_{\al}\|_{L^2})^{s},
\end{split}
\end{equation}
and
\begin{equation}\label{S6deguji}
\begin{split}
|S_6|\leq \ep\|\py\omega\|_{\mH^{s, \gamma}}^2
+C\ep^{-1}(1+\|\omega\|_{\mH^{s, \gamma}}+\|(1+y)^{\ga}g_{\al}\|_{L^2})^{s+1}.
\end{split}
\end{equation}

For $S_7$, we have
\begin{equation}\label{kuaiwanle}
\begin{split}
&\Big|\int_{\Om}(1+y)^{2\ga}g_{\al}
a\p_{\tau}^{\beta}(u-U)\px\p_{\tau}^{\al-\beta}U dxdy\Big|\\
\leq&\|(1+y)^{\ga}g_{\al}\|_{L^2}
\|(1+y)^{\ga-1}\p_{\tau}^{\beta}(u-U)\|_{L^2}
\|(1+y)a\|_{L^{\ty}}
\|\px\p_{\tau}^{\al-\beta}U\|_{L^{\ty}_x}\\
\leq&C \|(1+y)^{\ga}g_{\al}\|_{L^2}
\|u-U\|_{\mH^{s, \ga-1}}
\|\px\p_{\tau}^{\al-\beta}U\|_{H^1_x}\\
\leq&C M\|(1+y)^{\ga}g_{\al}\|_{L^2}
\|u-U\|_{\mH^{s, \ga-1}}.
\end{split}
\end{equation}
Hence, we can obtain
\begin{equation}\label{S7deguji}
\begin{split}
|S_7|
\leq C M\|(1+y)^{\ga}g_{\al}\|_{L^2}
\|u-U\|_{\mH^{s, \ga-1}}.
\end{split}
\end{equation}

For $S_8$, since $\om\in\mH^{s,\gamma}_{\si,\de}$, we have by \eqref{henyouyong}
\begin{equation*}
\begin{split}
|(1+y)a\om^{n-2}\py^2\om|=
&|(1+y)\om^{n-3}\py\om\py^2\om|\\
\leq
&\de^{n-4}|(1+y)^{1+\si(3-n)-\si-1-\si-2}|\\
\leq&\de^{n-4}|(1+y)^{-(n-1)\si-2}|\\
\leq &C.
\end{split}
\end{equation*}
Similarly, we can prove that
\begin{equation*}
\begin{split}
|(1+y)a\py^2\om^{n-1}|\leq C,\quad|(1+y)a^3\om^{n-1}|
\leq C.
\end{split}
\end{equation*}
Hence, by using H\"{o}lder's inequality, it gives
\begin{equation}\label{S8deguji}
\begin{split}
|S_8|
\leq C \|(1+y)^{\ga}g_{\al}\|_{L^2}
\|u-U\|_{\mH^{s, \ga-1}}.
\end{split}
\end{equation}

For $S_9$, it is easy to get
\begin{equation*}
\begin{split}
|\om^{n-1}\py a|\leq C (1+y)^{\si(1-n)-2}\leq C (1+y)^{-1}\leq C.
\end{split}
\end{equation*}
Then there holds
\begin{equation}\label{S9deguji}
\begin{split}
|S_9|
\leq C \|(1+y)^{\ga}g_{\al}\|_{L^2}^2.
\end{split}
\end{equation}

For $S_{10}$, it is easy to obtain by \eqref{ggg}
\begin{equation}\label{S10deguji}
\begin{split}
|S_{10}|
\leq &C \|(1+y)^{\ga}g_{\al}\|_{L^2}
\|(1+y)^{\ga}g_{e_2}\|_{L^2}
\|\ptau^{\al}U\|_{L^{\ty}_x}\\
\leq &C M \|(1+y)^{\ga}g_{\al}\|_{L^2}
(\|\om\|_{\mH^{s, \ga}}+\|u-U\|_{\mH^{s, \ga-1}}).
\end{split}
\end{equation}

Substituting all estimates of $S_i$ into \eqref{haoweixian} and summing over $\al$ , we can establish that \eqref{faguji} holds.

\hfill $\square$

$Proof~of~\eqref{ptaug}$: 
Applying $\py$ to \eqref{wodu}, we obtain
\begin{equation}\label{pywodu}
(\pt +u\px +v\py-n\om^{n-1}\py^2) \py\om
=-\om\px\om+\px u\py\om+
2n\py\om^{n-1}\py^2\om
+n\py^2\om^{n-1}\py\om.
\end{equation}
Then by \eqref{wodu} and \eqref{pywodu}, we can compute
\begin{equation}\label{a1}
\begin{split}
&(\pt +u\px +v\py-n\om^{n-1}\py^2) a\\
=&-n\om^{n-1}\py^2 a+
\frac{1}{\om}(\pt +u\px +v\py)\py\om
-\frac{\py\om}{\om^2}(\pt +u\px +v\py)\om\\
=&-n\om^{n-1}\py^2 a
+n\om^{n-2}\py^3\om-\px\om+a\px u+2n(n-1)a\om^{n-2}\py^2\om+
na\py^2\om^{n-1}\\
&-na\om^{n-2}\py^2\om-n(n-1)a^3\om^{n-1}.
\end{split}
\end{equation}
On the other hand, we can check that
\begin{equation}\label{a2}
\begin{split}
\py^2 a=\frac{\py^3\om}{\om}-a\frac{\py^2\om}{\om}-2a\py a.
\end{split}
\end{equation}
Substituting \eqref{a2} into \eqref{a1}, we get the following equation for $a$:
\begin{equation}\label{a3}
\begin{split}
&(\pt +u\px +v\py-n\om^{n-1}\py^2) a\\
=&
-g_{e_2}+a\px U+2n(n-1)a\om^{n-2}\py^2\om+
na\py^2\om^{n-1}\\
&-n(n-1)a^3\om^{n-1}+2na\py a\om^{n-1},
\end{split}
\end{equation}
where $g_{e_2}:=\ptau^{e_2}\om-a\ptau^{e_2}(u-U)=
\px\om-a\px(u-U)$.

Set $I=(\pt +u\px +v\py-n\om^{n-1}\py^2)$. Then Combining \eqref{ptauu}, \eqref{ptauwodu} and \eqref{a3} yields 
\begin{equation}\label{Ig}
\begin{split}
&I(g_{\al})\\
=&I(\ptau^{\al}\om)-aI(\ptau^{\al}(u-U))
-\ptau^{\al}(u-U)I(a)+2n\om^{n-1}\py a\py\ptau^{\al}(u-U)\\
=&-\sum_{0<\beta<\al}\binom{\al}{\beta}
\p_{\tau}^{\beta}u g_{\al-\beta+e_2}
-\sum_{0<\beta<\al}\binom{\al}{\beta}
\p_{\tau}^{\beta}v(\py \p_{\tau}^{\al-\beta}\om-a\p_{\tau}^{\al-\beta}\om)\\
&+n\sum_{0<\beta\leq\al}\binom{\al}{\beta}
\p_{\tau}^{\beta}\om^{n-1}(\py^2 \p_{\tau}^{\al-\beta}\om
-a\py \p_{\tau}^{\al-\beta}\om)\\
&+n(n-1)\sum_{0\leq\beta\leq\al}\binom{\al}{\beta}
\p_{\tau}^{\beta}\om^{n-2}\p_{\tau}^{\al-\beta}|\py\om|^2
+\sum_{0\leq\beta<\al}\binom{\al}{\beta}
a\p_{\tau}^{\beta}(u-U)\px\p_{\tau}^{\al-\beta}U\\
&-\ptau^{\al}(u-U)\{2n(n-1)a\om^{n-2}\py^2\om+na\py^2\om^{n-1}
-n(n-1)a^3\om^{n-1}\}\\
&+2n\om^{n-1}\py a g_{\al}-g_{e_2}\ptau U.
\end{split}
\end{equation}
Hence \eqref{ptaug} holds.

\hfill $\square$

$Proof~of~\eqref{S14}$: 
By the boundary conditions $\eqref{feiniu}_4$, we have
\begin{equation}\label{gabianjie}
\begin{split}
&(\om^{n-1} g_{\al}\py g_{\al} )|_{y=0}\\
=&
\{(\om^{n-1}(\ptau^{\al}\om+a\ptau^{\al}U)
(\ptau^{\al}\py\om-a\ptau^{\al}\om+\py a\ptau^{\al}U)\}|_{y=0}\\
=&\{\om^{n-1}\ptau^{\al}\om\ptau^{\al}\py\om
+a\om^{n-1}\ptau^{\al}U\ptau^{\al}\py\om
-a\om^{n-1}(\ptau^{\al}\om)^2\\
&-a^2\om^{n-1}\ptau^{\al}U\ptau^{\al}\om
+\py a\om^{n-1}\ptau^{\al}\om\ptau^{\al}U
+a\py a\om^{n-1}(\ptau^{\al}U)^2
\}|_{y=0}.
\end{split}
\end{equation}
First, since $\om\in\mH^{s,\gamma}_{\si,\de}$ for $\om$, it is easy to get there exists a $C_{\de}>0$, such that
\begin{equation}\label{youjie}
\begin{split}
|(a,\py a, \om^{n-1})|_{y=0}\leq C_{\de}.
\end{split}
\end{equation}
Then, using H\"{o}lder's inequality and \eqref{L1.5} yields that
\begin{equation}\label{jiandanbufen}
\begin{split}
&\Big|\int_{\T}\{-a\om^{n-1}(\ptau^{\al}\om)^2
-a^2\om^{n-1}\ptau^{\al}U\ptau^{\al}\om
+\py a\om^{n-1}\ptau^{\al}\om\ptau^{\al}U
+a\py a\om^{n-1}(\ptau^{\al}U)^2
\}|_{y=0}dx\Big|\\
\leq & C (\|\ptau^{\al}\om|_{y=0}\|^2_{L^2_x} + \|\ptau^{\al}U\|^2_{L^2_x})\\
\leq & C (\|\ptau^{\al}\om\|_{L^2}\|\ptau^{\al}\py\om\|_{L^2} + M^2)\\
\leq & \ep\|\py\omega\|_{\mH^{s, \gamma}}^2
+C\ep^{-1}(\|\om\|_{\mH^{s, \ga}}^2+
M^2).
\end{split}
\end{equation}
Next, by the boundary conditions $\eqref{wodu}_3$, we have
\begin{equation}\label{nandian}
\begin{split}
\om^{n-1}\ptau^{\al}\py\om|_{y=0}=
\frac{1}{n}\ptau^{\al}\px p- \frac{1}{n}\sum_{0<\beta\leq\al}\binom{\al}{\beta}
\ptau^{\beta}\om^{n-1}\ptau^{\al-\beta}\py\om|_{y=0}.
\end{split}
\end{equation}
Hence, putting \eqref{nandian} into \eqref{gabianjie} and using \eqref{L1}, similar to \eqref{K1435}, it yields 
\begin{equation}\label{sufu}
\begin{split}
\Big|\int_{\T}(\om^{n-1}\ptau^{\al}\om\ptau^{\al}\py\om
+a\om^{n-1}\ptau^{\al}U\ptau^{\al}\py\om)|_{y=0} dx\Big|\leq\ep\|\py\omega\|_{\mH^{s, \gamma}}^2
+C\ep^{-1}(1+M+\|\omega\|_{\mH^{s, \gamma}})^{2s}.
\end{split}
\end{equation}
Combining \eqref{jiandanbufen} and \eqref{sufu}, it shows  \eqref{S14} holds.

\hfill $\square$

\subsection{Closeness of the A Priori Estimates}\label{jianchi}

In this subsection, we are going to prove Proposition \ref{xianyanguji}. According to Proposition \ref{guji1} and \ref{guji2}, we obtain that from the definition of $\|\om\|_{\mH^{s, \gamma}_g}$
\begin{equation}\label{fb1}
\begin{split}
&\frac{d}{dt}\|\om\|_{\mH^{s, \gamma}_g}^2
+ n\de^{1-n}(\sum_{\substack{|\al|\leq s \\ \al_1+\al_2\leq s-1}}\|\py D^{\al}\om\|^2_{L^2_{\ga+\al_3}}
+ \sum_{\al_1+\al_2=s}\|\py g_{\al}\|^2_{L^2_{\ga}})
\\
\leq &C\ep(\|\py\omega\|_{\mH^{s, \gamma}}^2
+\sum_{\al_1+\al_2=s}\|\py g_{\al}\|^2_{L^2_{\ga}})+C\ep^{-1}(1+M+\|u-U\|_{\mH^{s, \ga-1}}\\
&\quad\quad\quad\quad\quad\quad\quad\quad\quad
\quad\quad\quad\quad\quad\quad\quad\quad\quad
+\|\omega\|_{\mH^{s, \gamma}}+\|\om\|_{\mH^{s, \gamma}_g})^{2s+2}.
\end{split}
\end{equation}
Since $g_{\al}:=\ptau^{\al}\om-a\ptau^{\al}(u-U)$ for any $\al=(\al_1,\al_2,0)$, we have
\begin{equation*}
\begin{split}
\py \ptau^{\al}\om=\py g_{\al}+ \py a \ptau^{\al}(u-U)+ a\ptau^{\al}\om.
\end{split}
\end{equation*}
Then
\begin{equation}\label{fb2}
\begin{split}
&\|(1+y)^{\ga}\py \ptau^{\al}\om\|_{L^2}\\
\leq&
\|(1+y)^{\ga}\py g_{\al}\|_{L^2}
+\|(1+y)^{\ga}\py a \ptau^{\al}(u-U)\|_{L^2}
+\|(1+y)^{\ga} a\ptau^{\al}\om\|_{L^2}\\
\leq&\|(1+y)^{\ga}\py g_{\al}\|_{L^2}
+\|(1+y)\py a\|_{L^{\ty}}
\|(1+y)^{\ga-1}\py a \ptau^{\al}(u-U)\|_{L^2}\\
&+\|a\|_{L^{\ty}}
\|(1+y)^{\ga} \ptau^{\al}\om\|_{L^2}\\
\leq&\|(1+y)^{\ga}\py g_{\al}\|_{L^2}
+C\|u-U\|_{\mH^{s, \ga-1}}
+C\|\omega\|_{\mH^{s, \gamma}}.
\end{split}
\end{equation}
Substituting \eqref{fb2} into \eqref{fb1}, we deduce
\begin{equation}\label{fb3}
\begin{split}
&\frac{d}{dt}\|\om\|_{\mH^{s, \gamma}_g}^2
+ n\de^{1-n}(\sum_{\substack{|\al|\leq s \\ \al_1+\al_2\leq s-1}}\|\py D^{\al}\om\|^2_{L^2_{\ga+\al_3}}
+ \sum_{\al_1+\al_2=s}\|\py g_{\al}\|^2_{L^2_{\ga}})
\\
\leq &C\ep(\sum_{\substack{|\al|\leq s \\ \al_1+\al_2\leq s-1}}\|\py D^{\al}\om\|^2_{L^2_{\ga+\al_3}}
+ \sum_{\al_1+\al_2=s}\|\py g_{\al}\|^2_{L^2_{\ga}})+C\ep^{-1}(1+M+\|u-U\|_{\mH^{s, \ga-1}}\\
&\quad\quad\quad\quad\quad\quad\quad\quad\quad
\quad\quad\quad\quad\quad\quad\quad\quad\quad
\quad\quad\quad\quad\quad\quad
+\|\omega\|_{\mH^{s, \gamma}}+\|\om\|_{\mH^{s, \gamma}_g})^{2s+2}.
\end{split}
\end{equation}
Hence, by choosing $\ep$ small enough, we get 
\begin{equation}\label{fb4}
\begin{split}
\frac{d}{dt}\|\om\|_{\mH^{s, \gamma}_g}^2
\leq C(1+M+\|u-U\|_{\mH^{s, \ga-1}}
+\|\omega\|_{\mH^{s, \gamma}}+\|\om\|_{\mH^{s, \gamma}_g})^{2s+2}.
\end{split}
\end{equation}
From \eqref{fb4} and \eqref{chabuduo}, we deduce
\begin{equation}\label{fb5}
\begin{split}
\frac{d}{dt}\|\om\|_{\mH^{s, \gamma}_g}^2
\leq& C(1+M+\|\om\|_{\mH^{s, \gamma}_g})^{2s+2}\\
\leq& C\|\om\|_{\mH^{s, \gamma}_g}^{2s+2}+
 C(1+M)^{2s+2}.
\end{split}
\end{equation}
Hence, using the comparison principle of ordinary differential equations gives 
\begin{equation}\label{fb6}
\begin{split}
\|\om\|_{\mH^{s, \gamma}_g}^2
\leq \{F(0)+
C(1+M)^{2s+2}t \}
\cdot\left\{
1-Cs\{F(0)+
C(1+M)^{2s+2}t\}^s t
\right\}^{-\frac{1}{s}},
\end{split}
\end{equation}
where
\begin{equation}\label{F0}
\begin{split}
F(0):=\sum_{\substack{|\al|\leq s \\ \al_1+\al_2\leq s-1}}
\|D^{\al}\om(0)\|^2_{L^2_{\ga+\al_3}}
+\sum_{\al_1+\al_2=s}
\|g_{\al}(0)\|^2_{L^2_{\ga}}.
\end{split}
\end{equation}

Next, by direct calculation, we know that $D^{\al}(\om,u-U)(0,x,y)$ with $|\al|\leq s$ can be expressed by spatial derivatives of initial data $(\om_0,u_0-U|_{t=0},U|_{t=0})$ up to order $2s$, therefore, by \eqref{chabuduo}, we get that there exists a polynomial $\mathcal{P}(\cdot)$, such that
\begin{equation}\label{chuzhiduoxiangshi}
F(0)\leq\mathcal{P}(M,\|u_0-U|_{t=0}\|_{H^{2s,\ga-1}},\|\om_0\|_{H^{2s,\ga}})
\end{equation}
Combining \eqref{fb6} and \eqref{chuzhiduoxiangshi} yields \eqref{zuihou1}.

Next, we also need to estimate $I(t):=\sum_{|\alpha| \leq 2}\left|(1+y)^{\sigma+\alpha_3} D^\alpha \omega\right|^2$ and $(1+y)^{\si}\om$.

We know that
\begin{equation*}
\begin{split}
(1+y)^{\si}\om(\tau)=(1+y)^{\si}\om_0+\int_{0}^{t}(1+y)^{\si}\pt \om(\tau) d\tau.
\end{split}
\end{equation*}
Then by $\om\in \mH^{s,\gamma}_{\si,\de}$, we obtain
\begin{equation}\label{omdebianjie}
\begin{split}
(1+y)^{\si}\om(\tau)\geq&(1+y)^{\si}\om_0-\int_{0}^{t}\|(1+y)^{\si}\pt \om(\tau)\|_{L^{\ty}} d\tau\\
\geq&(1+y)^{\si}\om_0-\int_{0}^{t}\de^{-1} d\tau\\
\geq&(1+y)^{\si}\om_0-\de^{-1}t,
\end{split}
\end{equation}
which shows \eqref{zuihou2}.

Now, for $I(t)$, we define $B_{\al}:=(1+y)^{\sigma+\al_3} D^\alpha \omega$ for $|\al|\leq2$. Then by a direct computation, we deduce that $B_{\al}$ satisfies 
\begin{equation}\label{bbll}
\begin{split}
\pt B_{\al} +u\px B_{\al} +v\py B_{\al}-n\om^{n-1}\py^2 B_{\al}
=
D_{\al}B_{\al}+E_{\al}\py B_{\al}+F_{\al},
\end{split}
\end{equation}
where
\begin{equation*}
\begin{split}
D_{\al}=(\si+\al_3)\frac{v}{1+y}
+n(\si+\al_3)(\si+\al_3+1)\frac{\om^{n-1}}{(1+y)^2},
\quad
E_{\al}=-2n(\si+\al_3)\frac{\om^{n-1}}{1+y},
\end{split}
\end{equation*}
and
\begin{equation*}
\begin{split}
F_{\al}=
&-(1+y)^{\si+\al_3}\sum_{0<\beta\leq\al}\binom{\al}{\beta}
\left(D^{\beta}u\px D^{\al-\beta}\om+D^{\beta}v\py D^{\al-\beta}\om\right)\\
&+n(1+y)^{\si+\al_3}\sum_{0<\beta\leq\al}\binom{\al}{\beta}
D^{\beta}\om^{n-1}\py^2 D^{\al-\beta}\om\\
&+n(n-1)(1+y)^{\si+\al_3}\sum_{0\leq\beta\leq\al}\binom{\al}{\beta}
D^{\beta}\om^{n-2}D^{\al-\beta}|\py\om|^2.
\end{split}
\end{equation*}
Then by \eqref{chabuduo}, \eqref{K121} and \eqref{vdeguji}, we have the following for $|\al|\leq 2$
\begin{equation}\label{DE}
\begin{split}
|D_{\al}|\leq C (1+M+\|\om\|_{\mH^{s, \gamma}_g}),\quad
|E_{\al}|\leq C\frac{\om^{n-1}}{1+y}.
\end{split}
\end{equation}
On the other hand, for $D^{\beta}u$, if $\beta_3>0$, then by \eqref{chabuduo}
\begin{equation*}
\begin{split}
|(1+y)^{\beta_3}D^{\beta}u|=&|(1+y)^{\beta_3}D^{\beta-e_3}\om|\\
\leq &|(1+y)^{\ga+\beta_3-1}D^{\beta-e_3}\om|
\leq C\|\omega\|_{\mH^{s, \gamma}}
\leq C(\|\om\|_{\mH^{s, \gamma}_g}+M).
\end{split}
\end{equation*}
If $|\beta_3|=0$, then by \eqref{L5}
\begin{equation*}
\begin{split}
|D^{\beta}u|=|\py^{-1}D^{\beta}\om|\leq C\|\omega\|_{\mH^{|\beta|+1, 1}}
\leq C(\|\om\|_{\mH^{s, \gamma}_g}+M)
\end{split}
\end{equation*}
Combining the above two cases, we have 
\begin{equation}\label{xixi}
\begin{split}
|(1+y)^{\beta_3}D^{\beta}u|
\leq C(\|\om\|_{\mH^{s, \gamma}_g}+M).
\end{split}
\end{equation}
For $D^{\beta}v$, similarly, we have 
\begin{equation}\label{xixiv}
\begin{split}
|(1+y)^{\beta_3}D^{\beta}v|
\leq C(\|\om\|_{\mH^{s, \gamma}_g}+M).
\end{split}
\end{equation}
Hence 
\begin{equation*}
\begin{split}
(1+y)^{\si+\al_3}|D^{\beta}u\px D^{\al-\beta}\om+D^{\beta}v\py D^{\al-\beta}\om|
\leq C(\|\om\|_{\mH^{s, \gamma}_g}+M)I^{\frac{1}{2}}(t).
\end{split}
\end{equation*}
Next, because of $|\al|\leq 2$, a simple case analysis suffices to prove 
\begin{equation*}
\begin{split}
(1+y)^{\si+\al_3}|D^{\beta}\om^{n-1}\py^2 D^{\al-\beta}\om|
\leq &C(\|\om\|_{\mH^{s, \gamma}_g}+1)I^{\frac{1}{2}}(t),\\
(1+y)^{\si+\al_3}|D^{\beta}\om^{n-2}D^{\al-\beta}|\py\om|^2|
\leq &C(\|\om\|_{\mH^{s, \gamma}_g}+1)I^{\frac{1}{2}}(t).
\end{split}
\end{equation*}
As a result, it gives for $|\al|\leq 2$
\begin{equation}\label{FFF}
\begin{split}
|F_{\al}|\leq C (1+M+\|\om\|_{\mH^{s, \gamma}_g})I^{\frac{1}{2}}(t).
\end{split}
\end{equation}
Noting that $I(t)=\sum_{|\alpha| \leq 2}|B_{\al}|^2$, we can therefore obtain by using \eqref{bbll}
\begin{equation}\label{iiii}
\begin{split}
&\pt I +u\px I +v\py I-n\om^{n-1}\py^2 I\\
=
&2 \sum_{|\alpha| \leq 2}(D_{\al}B_{\al}^2+2E_{\al}B_{\al}\py B_{\al}+2F_{\al}B_{\al})
-2n\om^{n-1}\sum_{|\alpha| \leq 2}|\py B_{\al}|^2.
\end{split}
\end{equation}
Then using \eqref{DE}, \eqref{FFF} and Cauchy's inequality yields
\begin{equation}\label{iiii2222}
\begin{split}
&\pt I +u\px I +v\py I-n\om^{n-1}\py^2 I\\
\leq & C (1+M+\|\om\|_{\mH^{s, \gamma}_g})I
+ C\frac{\om^{n-1}}{1+y}\sum_{|\alpha| \leq 2} B_{\al}\py B_{\al}
-2n\om^{n-1}\sum_{|\alpha| \leq 2}|\py B_{\al}|^2\\
\leq & C (1+M+\|\om\|_{\mH^{s, \gamma}_g})I
+ C\frac{\om^{n-1}}{(1+y)^2}\sum_{|\alpha| \leq 2} |B_{\al}|^2\\
\leq & C (1+M+\|\om\|_{\mH^{s, \gamma}_g})I.
\end{split}
\end{equation}
For the last inequality sign, we make use of $\frac{\om^{n-1}}{(1+y)^2}\leq C_{\de}$.

Using the classical maximum principle for parabolic equations (see Lemma \ref{jida}) to $I(t)$ yields
\begin{equation}\label{jidazhi}
\begin{split}
\|I(t)\|_{L^{\ty}}
\leq \max\left\{
e^{C (1+M+\|\om(t)\|_{\mH^{s, \gamma}_g})t}\|I(0)\|_{L^{\ty}},~
\max\limits_{\tau\in[0,t]}\{e^{C (1+M+\|\om(t)\|_{\mH^{s, \gamma}_g})(t-\tau)}\|I(\tau)|_{y=0}\|_{L^{\ty}(\T)}\}
\right\}.
\end{split}
\end{equation}

In inequalities \eqref{jidazhi}, although we have already controlled $I$ using initial and boundary value, since we do not have information about the boundary values, we next need to estimate it.

To derive \eqref{zuihou3}, by Lemma \ref{chazhi}, we have
\begin{equation}\label{bianjie11}
\begin{split}
\|I|_{y=0}\|_{L^{\ty}(\T)}\leq&
\sum_{|\alpha| \leq 2}\|D^{\al}\om\|_{L^{\ty}}^2\\
\leq
&3\tilde{C}^2\sum_{|\alpha| \leq 2}(\|D^{\al}\om\|_{L^2}^2
+\|\px D^{\al}\om\|_{L^2}^2
+\|\py^2 D^{\al}\om\|_{L^2}^2)\\
\leq&6\tilde{C}^2 \|\om\|_{\mH^{s, \gamma}_g}^2.
\end{split}
\end{equation}

To derive \eqref{zuihou4}, for $s\geq5$, we have by Lemma \ref{chazhi} that 
\begin{equation}\label{bianjie22}
\begin{split}
\|I(t)|_{y=0}\|_{L^{\ty}(\T)}\leq
&\|I(0)|_{y=0}\|_{L^{\ty}(\T)}+\int_{0}^{t}\|\pt I(\tau)|_{y=0}\|_{L^{\ty}(\T)}d\tau\\
\leq&\|I(0)\|_{L^{\ty}}+2\sum_{|\alpha| \leq 2}\int_{0}^{t}\|\pt D^{\al}\om D^\alpha \omega(\tau)\|_{L^{\ty}} d\tau\\
\leq&\|I(0)\|_{L^{\ty}}+2\tilde{C}^2\sum_{|\alpha| \leq 2}\int_{0}^{t}(\|D^{\al}\om\|_{L^2}
+\|\px D^{\al}\om\|_{L^2}
+\|\py^2 D^{\al}\om\|_{L^2})\\
&\quad\quad\quad\quad\quad\quad\quad\quad\quad\quad
\cdot(\|\pt D^{\al}\om\|_{L^2}
+\|\pt \px D^{\al}\om\|_{L^2}
+\|\pt \py^2 D^{\al}\om\|_{L^2}) d\tau\\
\leq&\|I(0)\|_{L^{\ty}}+4\tilde{C}^2\int_{0}^{t}\|\om(\tau)\|_{\mH^{s, \gamma}_g}^2 d\tau,
\end{split}
\end{equation}
which shows \eqref{zuihou4}.

This completes the proof of Proposition \ref{xianyanguji}. With this proposition in hand, one can prove Theorem \ref{zhuding} by applying the standard regularization procedure described in \cite[Sec. 4]{CJLCPAM} (see also \cite{NM}); we omit the details here.

\section{An alternative norm}\label{lingyizhong}
In this section, we introduce another function space to solve the pseudo-plastic fluid boundary layer equations \eqref{feiniu}. 
Denote derivative (only in space) operator by
$$D_{xy}^{\al}=\px^{\al_1}\py^{\al_2}
~~\text{for}~~\al=(\al_1,\al_2)\in\N^2,~~
|\al|=\al_1+\al_2.$$
Then define $H^{s,\gamma}_{\si,\de}$ 
$$
\begin{aligned}
H_{\sigma, \delta}^{s, \gamma}:= & \bigg\{\omega: \Om \rightarrow \mathbb{R}:\|\omega\|_{H^{s, \gamma}}<+\infty,(1+y)^\sigma \omega \geq \delta, \\
& \text { and } \sum_{|\al|\leq 2}\left|(1+y)^{\sigma+\alpha_2} D_{xy}^{\al}\omega\right|^2 \leq \frac{1}{\delta^2}\bigg\},
\end{aligned}
$$
where $s\geq4,~\ga\geq1,~\si>\ga+\frac{1}{2},~\de\in(0,1)$.

Moreover, we define $\|\om\|_{H^{s, \gamma}_g}$ accordingly,
$$
\|\om\|_{H^{s, \gamma}_g}^2=
\sum_{\substack{|\al|\leq s \\ \al_1\leq s-1}}
\|D_{xy}^{\al}\om\|^2_{L^2_{\ga+\al_2}}
+
\|g_{s}\|^2_{L^2_{\ga}},
$$
where $g_{s}:=\px^s\om-a\px^s(u-U)$ and $a:=\frac{\py\om}{\om}$. 

Then, using the above function spaces, we can obtain a result similar to \cite[Theorem 2.2]{NM}, as shown below:
\begin{Theorem}\label{zhuding2}
Assume $\frac{1}{3}<n<1$. Let $s\geq4$ be an even integer, $\ga\geq1,~\ga+\frac{1}{2}<\si
\leq\min\{\frac{1}{1-n},\ga+\frac{2}{3}\}$ and
$\de\in(0,\frac{1}{2})$. And we suppose the outer flow $U$ satisfies
\begin{equation}\label{wailiushangjie2}
M:=\sup\limits_{t}\sum\limits_{i=0}^{[\frac{s+9}{2}]}
\|\pt^i U\|_{W^{s-2i+9,\ty}(\T)}<+\ty.
\end{equation}
Assume that the initial velocity  $u_0-U|_{t=0}\in H^{s,\ga-1}$ and the initial vorticity $\om_0:=\py u_0\in H_{\sigma, 2\delta}^{s, \gamma}$.
In addition, when $s=4$, we further assume that $\de>0$ is chosen small enough such that
$$
\|\om_0\|_{H^{s, \gamma}_g}\leq \frac{1}{192\tilde{C}}\de^{-1},
$$
where the universal constant $\tilde{C}$ is the same as the in Lemma \ref{chazhi}.
Then there exist a time $T>0$ and a unique solution $(u,v)$ to the pseudo-plastic fluid boundary layer equations \eqref{feiniu} such that
\begin{equation*}
  u-U\in L^{\ty}([0,T];H^{s,\ga-1})\cap C([0,T];H^s-w)
\end{equation*}
and the vorticity 
\begin{equation*}
  \om:=\py u\in L^{\ty}([0,T];H_{\sigma, \delta}^{s, \gamma})\cap C([0,T];H^s-w),
\end{equation*}
where $H^s-w$ denotes the space $H^s$ endowed with its weak topology. 

Furthermore, the vorticity $\om$ satisfies all estimates stated in Proposition \ref{xianyanguji222}.
\end{Theorem}
The proof of Theorem \ref{zhuding2} relies on the following a priori estimates. 
\begin{Proposition}\label{xianyanguji222}
Assume $\frac{1}{3}<n<1$. Let $s\geq4$ be an even integer, $\ga\geq1,~\ga+\frac{1}{2}<\si
\leq\min\{\frac{1}{1-n},\ga+\frac{2}{3}\}$ and
$\de\in(0,\frac{1}{2})$.
Assume all the hypotheses for $U$ and $(u_0,\om_0)$ given in Theorem \ref{zhuding2} hold.
If $(u,v,\om)$ is a classical solution of \eqref{wodu} in $[0,T]$ and satisfies
\begin{equation}\label{jiasheeee}
\om\in L^{\ty}(0,T;H^{s,\gamma}_{\si,\de}),~\py\om\in L^{2}(0,T;H^{s,\gamma}).
\end{equation}
Then, there exists a positive constant $C$, depending only on $s,~\ga,~\si,~\de$ and $n$, such that for small time,
\begin{equation}\label{zuihou1111}
\begin{split}
\|\om\|_{H^{s, \gamma}_g}^2
\leq Q(t), 
\end{split}
\end{equation}
where
\begin{equation}\label{qtttt}
\begin{split}
Q(t):=\{\|\om_0\|_{H^{s, \gamma}_g}^2+
C(1+M)^{2s+2}t \}
&\cdot\left\{
1-Cs\{\|\om_0\|_{H^{s, \gamma}_g}^2+
C(1+M)^{2s+2}t\}^s t
\right\}^{-\frac{1}{s}}
\end{split}
\end{equation}

Moreover, define
\begin{equation}\label{Atttt}
\begin{split}
A(t):=\max_{[0,t]}\|\om\|_{H^{s, \gamma}_g}.
\end{split}
\end{equation}

Then, we have that
\begin{equation}\label{zuihou2222}
\begin{split}
\min\limits_{\Omega}(1+y)^{\si}\om\geq
&\left\{1-C (1+M+A(t))t e^{C (1+M+A(t))t}\right\}\\
&\cdot\left\{
\min\limits_{\Omega}(1+y)^{\si}\om_0-C(1+A(t))A(t)t
\right\},
\end{split}
\end{equation}
and
\begin{equation}\label{zuihou3333}
\begin{split}
\|I(t)\|_{L^{\ty}}
\leq \max\left\{
\|I(0)\|_{L^{\ty}},~
6\tilde{C}^2A^2(t)\right\}e^{C (1+M+A(t))t}
,
\end{split}
\end{equation}
where $I(t):=\sum_{|\alpha| \leq 2}\left|(1+y)^{\sigma+\alpha_2} D_{xy}^\alpha \omega\right|^2$ and the universal constant $\tilde{C}$ is the same as the in Lemma \ref{chazhi}.

In addition, if $s\geq6$, then there also holds
\begin{equation}\label{zuihou4444}
\begin{split}
\|I(t)\|_{L^{\ty}}
\leq \left\{
\|I(0)\|_{L^{\ty}}+C(1+A(t))A^2(t)t
\right\}e^{C (1+M+A(t))t}
.
\end{split}
\end{equation}
\end{Proposition}
Based on Proposition \ref{xianyanguji222}, a standard regularization procedure yields the local-in-time existence and uniqueness. The reader may refer to \cite{NM} for a detailed account of the regularization procedure.

\textbf{Proof of Proposition \ref{xianyanguji222}.}
The proof of Proposition \ref{xianyanguji222} is similar to that of Proposition \ref{xianyanguji}, with the main differences lying in two aspects: the estimation of boundary term and the derivation of \eqref{zuihou2222}-\eqref{zuihou4444}. For the sake of brevity, we do not present the complete proof here, but only provide the treatment of the differences.

For the derivation of \eqref{zuihou2222}-\eqref{zuihou4444}:
we should use the classical maximum principle for parabolic equations to prove \eqref{zuihou2222}-\eqref{zuihou4444}, as is done in \cite{NM}.

For the estimation of boundary term:
In Section \ref{zhumingti}, we directly use the equation to reduce the order of the boundary terms (see \eqref{K146}); however, here, because the definition of the norm does not include time derivatives, we cannot use this method. To this end, we need to derive precise boundary conditions, as shown below:
\begin{Lemma}\label{bianjiejiangjie}
We have the following at boundary $y=0$,
\begin{equation}\label{yijie}
\py\om|_{y=0}=\frac{1}{n}\om^{1-n}|_{y=0}\px p.
\end{equation}
And for any $1\leq k\leq \frac{s}{2}$, there holds
\begin{equation}\label{gaojie}
\begin{split}
\py^{2k+1}\om|_{y=0}=&\frac{1}{n^{k+1}}\om^{k+1-(k+1)n}|_{y=0}\pt^k\px p\\
&+\sum_{j=2}^{2k+1}\sum_{\rho\in A_k^j}
Q_{k,\rho}(\om,\px p,\pt\px p,\cdots,\pt^{k-1}\px p)
\prod_{i=1}^{j}D^{\rho^i}\om|_{y=0},
\end{split}
\end{equation}
where $Q_{k,\rho}(\om,\px p,\pt\px p,\cdots,\pt^{k-1}\px p)$ denotes a polynomial in $(\px p,\pt\px p,\cdots,\pt^{k-1}\px p)$ of degree at most $k$, whose coefficients may be constants or powers of $\om$.
And $A_k^j:=\{\rho:=(\rho^1,\rho^2,\cdots,\rho^j)\in \mathbb{N}^{2j};1\leq|\rho^i|=\rho^i_1+\rho^i_2\leq 2k;
\sum_{i=1}^{j}|\rho^i|\leq 2k+1
\}$.
\end{Lemma}

\hfill $\square$

\textbf{Proof of Lemma \ref{bianjiejiangjie}.}
For \eqref{yijie}, it is exactly the same as the boundary condition $\eqref{wodu}_3$. In order to derive \eqref{gaojie}, first, we rewrite $\eqref{wodu}_1$ as
\begin{equation}\label{gaixie}
\pt f + u\px f + v\py f - \py^2\om +g=0,
\end{equation}
where
\begin{equation}\label{fheg}
f=\frac{1}{n(2-n)}\om^{2-n},\quad g=(1-n)\om^{-1}|\py\om|^2.
\end{equation}
Hence differentiating \eqref{gaixie} with respect to $y$ and then evaluating at $y=0$, we obatin
\begin{equation}\label{xinleilei}
\pt \py f|_{y=0} + \om\px f|_{y=0}  - \py^3\om|_{y=0} +\py g|_{y=0}=0.
\end{equation}
On the other hand, we have
\begin{equation}\label{xinleilei2}
\begin{split}
\pt \py f|_{y=0}=&\pt(\frac{1}{n}\om^{1-n}\py\om)|_{y=0} \\
=&\frac{1-n}{n}\om^{-n}\py\om\pt\om|_{y=0}
+\frac{1}{n}\om^{1-n}\pt\py\om|_{y=0}\\
=&\frac{1-n}{n}\om^{-n}\py\om\{n\om^{n-1}\py^2\om+n(n-1)\om^{n-2}
|\py\om|^2\}|_{y=0}\\
&+\frac{1}{n}\om^{1-n}(\frac{1-n}{n}\om^{-n}\pt\om\px p
+\frac{1}{n}\om^{1-n}\pt\px p)|_{y=0}\\
=&\frac{1-n}{n}\om^{-n}\py\om\{n\om^{n-1}\py^2\om+n(n-1)\om^{n-2}
|\py\om|^2\}|_{y=0}\\
&+\frac{1-n}{n^2}\om^{1-2n}\{n\om^{n-1}\py^2\om+n(n-1)\om^{n-2}
|\py\om|^2\}|_{y=0}\px p\\
&+\frac{1}{n^2}\om^{2-2n}|_{y=0}\pt\px p.
\end{split}
\end{equation}
Substituting \eqref{fheg} and \eqref{xinleilei2} into \eqref{xinleilei}, we justify the formula \eqref{gaojie} for $k=1$.

Next, we will prove formula \eqref{gaojie} by induction on $k$. For notation convenience, we denote
$$
\mathcal{A}_k:=\left\{\sum_{j=2}^{2k+1}\sum_{\rho\in A_k^j}
Q_{k,\rho}(\om,\px p,\pt\px p,\cdots,\pt^{k-1}\px p)
\prod_{i=1}^{j}D_{xy}^{\rho^i}\om|_{y=0}\right\}.
$$
Under this notation, we need to prove $\py^{2k+1}\om|_{y=0}-\frac{1}{n^{k+1}}\om^{k+1-(k+1)n}|_{y=0}\pt^k\px p\in\mathcal{A}_k$.

Assuming that \eqref{gaojie} holds for $k=m$, next, we consider $k=m+1$. Differentiating \eqref{gaixie} with respect to $y$ $2m+1$ times and then evaluating at $y=0$, it yields
\begin{equation}\label{gaojiegaixie}
\begin{split}
\py^{2m+3}\om|_{y=0}=
&\pt \py^{2m+1}f|_{y=0} +\py^{2m+1}g|_{y=0}\\
&+ \sum_{l=1}^{2m+1}\binom{2m+1}{l}\py^{l-1}\om\px\py^{2m+1-l} f|_{y=0}\\ 
&- \sum_{l=2}^{2m+1}\binom{2m+1}{l}\px\py^{l-2}\om\py^{2m+2-l} f|_{y=0}
.
\end{split}
\end{equation}
Obviously, using \eqref{fheg} and Lemma \ref{faa}, one can directly verify that the last three terms belong to $\mathcal{A}_{m+1}$. Hence, we only deal with the term $\pt \py^{2m+1}f|_{y=0}$.

First, by the induction hypothesis, we have the following 
\begin{equation}\label{gaojiem}
\begin{split}
\py^{2m+1}\om|_{y=0}=&\frac{1}{n^{m+1}}\om^{m+1-(m+1)n}|_{y=0}\pt^m\px p\\
&+\sum_{j=2}^{2m+1}\sum_{\rho\in A_m^j}
Q_{m,\rho}(\om,\px p,\pt\px p,\cdots,\pt^{m-1}\px p)
\prod_{i=1}^{j}D_{xy}^{\rho^i}\om|_{y=0}.
\end{split}
\end{equation}
And by Lemma \ref{faa}, it gives
\begin{equation}\label{haonanshou}
\py^{2m+1} f =
\frac{1}{n}\om^{1-n}\py^{2m+1}\om
+
 \sum_{r=2}^{2m+1}\om^{2-n-r} \sum_{\substack{\eta^1 + \cdots + \eta^r = 2m+1 \\ \eta^i \ge 1}}  C_{n,m,r,\eta^1,\cdots,\eta^r}\prod_{i=1}^r \py^{\eta^i}\om.
\end{equation}
Then differentiating \eqref{haonanshou} with respect to $t$ and evaluating at $y=0$, we obtain
\begin{equation}\label{haonanshou2}
\begin{split}
\pt\py^{2m+1} f|_{y=0}
=&\frac{1}{n}\om^{1-n}\pt\py^{2m+1}\om|_{y=0}
+\frac{1-n}{n}\om^{-n}\pt\om\py^{2m+1}\om|_{y=0}\\
&+\sum_{r=2}^{2m+1}\om^{2-n-r} \sum_{\substack{\eta^1 + \cdots + \eta^r = 2m+1 \\ \eta^i \ge 1}} C_{n,m,r,\eta^1,\cdots,\eta^r}\sum_{i_0=1}^{r} \pt\py^{\eta^{i_0}}\om\prod_{i\neq i_0}\py^{\eta^i}\om|_{y=0}\\
&+ \pt\om\sum_{r=2}^{2m+1}\om^{1-n-r}\sum_{\substack{\eta^1 + \cdots + \eta^r = 2m+1 \\ \eta^i \ge 1}}  \tilde{C}_{n,m,r,\eta^1,\cdots,\eta^r}\prod_{i=1}^r \py^{\eta^i}\om|_{y=0}
 .
\end{split}
\end{equation}
Now, by \eqref{gaojiem}, we can obtain
\begin{equation}\label{niuma}
\begin{split}
&\frac{1}{n}\om^{1-n}\pt\py^{2m+1}\om|_{y=0}\\
=
&\frac{1}{n^{m+2}}\om^{m+2-(m+2)n}|_{y=0}\pt^{m+1}\px p
+\frac{m+1-(m+1)n}{n^{m+2}}\om^{m+1-(m+2)n}\pt\om|_{y=0}\pt^m\px p
\\
&+\pt\om\sum_{j=2}^{2m+1}\sum_{\rho\in A_m^j}
\tilde{Q}_{m,\rho}(\om,\px p,\pt\px p,\cdots,\pt^{m-1}\px p)
\prod_{i=1}^{j}D_{xy}^{\rho^i}\om|_{y=0}\\
&+\sum_{j=2}^{2m+1}\sum_{\rho\in A_m^j}
\tilde{\tilde{Q}}_{m,\rho}(\om,\px p,\pt\px p,\cdots,\pt^{m-1}\px p,\pt^{m}\px p)
\prod_{i=1}^{j}D_{xy}^{\rho^i}\om|_{y=0}\\
&+\sum_{j=2}^{2m+1}\sum_{\rho\in A_m^j}
Q_{m,\rho}(\om,\px p,\pt\px p,\cdots,\pt^{m-1}\px p)
\sum_{i_0=1}^{j}\pt D_{xy}^{\rho^{i_0}}\om\prod_{i\neq i_0}D_{xy}^{\rho^i}\om|_{y=0}.
\end{split}
\end{equation}
Combining \eqref{gaojiegaixie}-\eqref{niuma} and the equation satisfied by $\pt D_{xy}^{\rho^{i_0}}\om$ \eqref{alwodu} yields $$\py^{2m+3}\om|_{y=0}-\frac{1}{n^{m+2}}\om^{m+2-(m+2)n}|_{y=0}\pt^{m+1}\px p\in\mathcal{A}_{m+1},$$ which shows Lemma \ref{bianjiejiangjie} holds.

\hfill $\square$

\subsection*{Acknowledgements}
Zhonger Wu is supported by STU Scientific Research Initiation Grant (NTF25028T). Zhong Tan is supported by the National Natural Science Foundation of China, NSFC (No. 12071391, No. 12231016) and Guangdong Basic and Applied Basic Research Foundation (No. 2022A1515010860).

\subsection*{Competing interests}

This work does not have any conflicts of interest.

\end{document}